\documentclass[article,ij4uq]{ij4uq}      

\usepackage[hang]{footmisc}
\usepackage[normalem]{ulem}
\fancypagestyle{plain}{%
  \fancyhf{}
  \fancyhead[R]{\small {\it International Journal for Uncertainty Quantification}, x(x): \thepage--\pageref{LastPage} (\myyear\today)}
  \fancyfoot[R]{\small\bf\thepage }
  \fancyfoot[L]{\fottitle}
  }
\renewcommand{\dmyy}{22}
\renewcommand{\myyear}{2022}
\renewcommand{\today}{}
\begin{document}

\volume{Volume x, Issue x, \myyear\today}
\title{Shapley effect estimation in reliability-oriented sensitivity analysis with correlated inputs by importance sampling}
\titlehead{Reliability-oriented Shapley effect estimation}
\authorhead{J. Demange-Chryst, F. Bachoc \& J. Morio}
\corrauthor[1,2]{Julien Demange-Chryst}
\author[1]{François Bachoc}
\author[2]{Jérôme Morio}
\corremail{julien.demange-chryst@onera.fr}
\address[1]{Institut de Mathématiques de Toulouse, UMR5219 CNRS, 31062 Toulouse, France}
\address[2]{ONERA/DTIS, Université de Toulouse, F-31055 Toulouse, France}

\dataO{03/17/2022}
\dataF{mm/dd/yyyy}

\abstract{Reliability-oriented sensitivity analysis aims at combining both reliability and sensitivity analyses by quantifying the influence of each input variable of a numerical model on a quantity of interest related to its failure. In particular, target sensitivity analysis focuses on the occurrence of the failure, and more precisely aims to determine which inputs are more likely to lead to the failure of the system. The Shapley effects are quantitative global sensitivity indices which are able to deal with correlated input variables. They have been recently adapted to the target sensitivity analysis framework. In this article, we investigate two importance-sampling-based estimation schemes of these indices which are more efficient than the existing ones when the failure probability is small. Moreover, an extension to the case where only an i.i.d. input/output $N$-sample distributed according to the importance sampling auxiliary distribution is proposed. This extension allows to estimate the Shapley effects only with a data set distributed according to the importance sampling auxiliary distribution stemming from a reliability analysis without additional calls to the numerical model. In addition, we study theoretically the absence of bias of some estimators as well as the benefit of importance sampling. We also provide numerical guidelines and finally, realistic test cases show the practical interest of the proposed methods.}


\maketitle

\section{Introduction}
More and more physical phenomenons and complex systems are numerically represented by black-box numerical models, which are often computationally expensive to evaluate, and whose complexity makes it impossible to study analytically. For safety and certification purposes, tracking the potential failures of a system is crucial, but it is not an option to do so experimentally with critical systems because it could lead to dramatic environmental, human or financial consequences. Numerical models enable to simulate the behaviour of a system far from nominal configurations.

The \textit{reliability analysis} of a numerical model mainly consists in the estimation of its failure probability. The failure is often a rare event and thus has a small probability. The high computational cost of one evaluation of the numerical model (several minutes to several days CPU) and the low value of the probability make the usual quadrature methods \cite{davis2007methods} and Monte Carlo sampling \cite{rubinstein2016simulation} inappropriate to handle this problem, but various techniques reviewed in \cite{morio2015estimation} have been developed to estimate more precisely a such probability at a limited computational cost, including importance sampling \cite{bucklew2004introduction} for example.

\textit{Global sensitivity analysis} (GSA) aims at studying the impact of the input variables of a numerical code on the behaviour of its output to provide a better understanding of the model. It can be carried out for various purposes such as fixing non influential input variables to nominal values or identifying the most influential ones to decrease their variability, see for example \cite{saltelli2004sensitivity}. A deeper analysis of the failure of the system consists then in combining both reliability and sensitivity analyses by performing a GSA on a quantity of interest (QoI) characterizing the failure of the system. This new specific framework is called \textit{reliability-oriented sensitivity analysis} (ROSA) and can be divided into two categories \cite{MARREL2021107711}: \begin{itemize}
    \item \textit{target sensitivity analysis} (TSA) aims at determining the influence of each input variable on the occurrence of the failure of the system
    \item \textit{conditional sensitivity analysis} (CSA) aims at performing a global sensitivity analysis of the numerical code restricted to the failure domain.
\end{itemize} In the present article, we only focus on TSA. The well-known Sobol indices for variance-based GSA \cite{sobol1993sensitivity} have been adapted to TSA and first specific estimation schemes have been introduced in \cite{wei2012efficient,perrin2019efficient}. However, as in GSA, the interpretability of these indices requires the strong assumption of independent input variables. Several approaches have been investigated in GSA to adapt the Sobol indices to the case where the inputs are correlated \cite{chastaing2012generalized} but recently, new GSA indices based on game theory \cite{shapley1953value} and which are more naturally able to deal with correlated inputs have been introduced: the \textit{Shapley effects} \cite{owen2014sobol}. Their adaptation to the TSA framework is recent \cite{ILIDRISSI2021105115} and the authors proposed first estimation schemes based on a classical Monte Carlo sampling according to the input distribution.

As illustrated in our numerical simulations in \Cref{sec:numerical_experiments}, the existing estimators of the Shapley effects for TSA based on a Monte Carlo sampling from \cite{ILIDRISSI2021105115} are not efficient when the failure probability is small because they require too many calls to the numerical code to be accurate. In this article, we introduce then new importance-sampling-based estimators of these indices which are able to deal more efficiently with a small failure probability. Moreover, we extend these new estimators to the case where only an i.i.d. input/output $N$-sample distributed according to the importance sampling auxiliary distribution is available, using the nearest neighbour approximation described in \cite{broto2020variance}. A major practical advantage is that our extended estimators enable to estimate efficiently the Shapley effects for TSA without additional calls to the function after the estimation of the failure probability by importance sampling. In addition, under the condition that the reliability analysis has been done efficiently, we show theoretically that the proposed estimators improve the estimation of the Shapley effects for TSA compared to the existing ones when the failure probability is getting smaller and finally, we give some numerical guidelines. 

The remainder of this paper is organized as follows. First, \Cref{sec:review} consists in a review on variance-based global sensitivity analysis, importance sampling and reliability-oriented sensitivity analysis. Then, \Cref{sec:contribution} introduces and describes the proposed importance-sampling-based estimators of the Shapley effects for TSA. In addition, \Cref{sec:numerical_experiments} illustrates the practical interest of the new estimators on numerical examples: the Gaussian linear case, a cantilever beam problem and a fire spread model. Finally, \Cref{sec:conclusion} concludes the present article and gives future research perspectives stemming from this work. 

\section{A review on global sensitivity indices: definitions, estimation schemes and adaptation to reliability}
\label{sec:review}

In this section, we recall the main principle of variance-based GSA and we describe very common existing sensitivity indices as well as some of their estimation schemes proposed in the literature. Next, after a brief reminder of importance sampling, we also review some tools from ROSA.

First of all, let us begin by introducing the notations that will be used throughout the paper. We let $\mathbf{X} = \left(X_1,\dots,X_d\right)$ be the input random vector on the input domain $\mathbb{X} = \bigotimes_{i=1}^d \mathbb{X}_i \subseteq \mathbb{R}^d$ with joint PDF $f_{\mathbf{X}}$. Then, the black-box function is defined by: \begin{equation}
    \begin{array}{l|rcl}
\phi : & \mathbb{X} & \longrightarrow & \mathbb{R} \\
    & \mathbf{x} & \longmapsto & y = \phi\left(\mathbf{x}\right).
    \end{array} 
\end{equation} 
No regularity hypothesis on $\phi$ is required but the random output $Y = \phi\left(\mathbf{X}\right) \in \mathbb{R}$ is supposed to be square integrable, i.e. $\mathbb{E}\left(Y^2\right) < + \infty$. Moreover, we let $\mathcal{P}(d) = \lbrace u \subseteq [\![1,d]\!] \rbrace$ denote all the subsets of $[\![1,d]\!] = \lbrace 1,\dots,d\rbrace$. Then, for any $u \in \mathcal{P}(d)$, let us write $-u = [\![1,d]\!] \backslash u$ for the complementary of the set $u$. In particular, for all $i \in [\![1,d]\!]$, $-i$ refers to the subset $[\![1,d]\!] \backslash \lbrace i \rbrace$. In addition, for any non-empty subset $u \in \mathcal{P}(d)$, letting $u = \lbrace i_1,\dots,i_r\rbrace$ with $i_1<\dots<i_r$, let $\mathbb{X}_u = \bigotimes_{j=1}^r \mathbb{X}_{i_j} \subseteq \mathbb{R}^{r}$ be the input domain of the random sub-vector $\mathbf{X}_u = \left(X_{i_1},\dots,X_{i_r} \right) = \left(X_i\right)_{i\in u}$. Furthermore, for $u \in \mathcal{P}(d)\backslash\lbrace \varnothing, [\![1,d]\!]\rbrace$, for any $\mathbf{x}_u \in \mathbb{X}_u$ and for any $\mathbf{x}_{-u} \in \mathbb{X}_{-u}$, $\left(\mathbf{x}_u,\mathbf{x}_{-u}\right)$ represents the vector $\tilde{\mathbf{x}} \in \mathbb{X}$ such that $\tilde{\mathbf{x}}_u = \mathbf{x}_u$ and $\tilde{\mathbf{x}}_{-u} = \mathbf{x}_{-u}$. We also write $\phi\left(\mathbf{x}_u,\mathbf{x}_{-u}\right) = \phi\left(\tilde{\mathbf{x}}\right)$. Finally, for any probability density $g : \mathbb{X}\longrightarrow\mathbb{R}_+$, we let $\mathbb{E}_g$ and $\mathbb{V}_g$ denote respectively the expectation and the variance operators of a random variable distributed according to the law of PDF $g$. When there is no ambiguity, we may also write $\mathbb{E}$ and $\mathbb{V}$ for $\mathbb{E}_{f_{\mathbf{X}}}$ and $\mathbb{V}_{f_{\mathbf{X}}}$.


\subsection{From Sobol indices to Shapley effects}
\label{subsec:2.1}

The Hoeffding functional decomposition \cite{10.1214/aoms/1177730196} allows to represent a function defined on any subset of $\mathbb{R}^d$ as a sum of elementary functions. When considering an input measure with independent components, this decomposition is unique under orthogonality conditions stated by \cite{sobol1993sensitivity}. Then, in the sensitivity analysis framework with independent inputs in $\mathbf{X}$ and a square integrable random output on the form $Y = \phi\left(\mathbf{X}\right)$, this decomposition leads to a unique functional decomposition of the variance of $Y$, also called ANOVA (ANalysis Of VAriance):\begin{equation}\label{anova}\mathbb{V}\left(Y\right) = \mathbb{V}\left(\phi(\mathbf{X})\right) = \sum_{u \in \mathcal{P}(d)\backslash\lbrace\varnothing\rbrace} \mathbb{V}\left(\phi_u(\mathbf{X}_u)\right),\end{equation} where for $u \in \mathcal{P}(d)\backslash\lbrace\varnothing\rbrace, \ \phi_u(\mathbf{X}_u) = \mathbb{E}\left(\phi(\mathbf{X}) | \mathbf{X}_u\right) + \sum_{v\subsetneq u} (-1)^{|u| - |v|} \mathbb{E}\left(\phi(\mathbf{X}) | \mathbf{X}_v\right)$. Then, the well-known \textit{first-order Sobol indices} \cite{sobol1993sensitivity} for GSA are obtained for all $i \in [\![1,d]\!]$ by: \begin{equation}\label{sobol}
    S_i = \dfrac{\mathbb{V}\left[\mathbb{E}\left(\phi\left(\mathbf{X}\right)|X_i\right)\right]}{\mathbb{V}\left(\phi\left(\mathbf{X}\right)\right)} \in [0,1].
\end{equation}
The index $S_i$ quantifies the part of variance of the output explained only by the input $X_i$. From \eqref{anova}, it is also possible to define higher-order Sobol indices which take into account the interactions between the input variables in $\phi$ but their number increases exponentially with the dimension $d$ and evaluating all of them becomes impossible. Thus, instead of computing higher-order indices, one typically prefers to consider another family of indices, the \textit{total-order Sobol indices} introduced by \cite{homma1996importance}, 
where each of them quantifies the part of variance of the output explained by the input $X_i$ in interaction with any other group of variables. In practice, when $d$ is large, the first-order and the total-order indices give satisfying information for the sensitivity analysis of the model.

When the inputs are correlated, the Hoeffding decomposition is no longer unique and even if it is still possible to define and compute Sobol indices, they don't allow to clearly identify the origin of the variability of the output anymore. To address this issue, by an analogy between game theory \cite{shapley1953value} and GSA, the author of \cite{owen2014sobol} introduced new variance-based sensitivity indices which are able to deal with correlated inputs called \textit{Shapley effects} or \textit{Shapley values}, defined for all $i \in [\![1,d]\!]$ by: \begin{equation}\label{shapley}
    \text{Sh}_i = \dfrac{1}{\mathbb{V}\left(\phi\left(\mathbf{X}\right)\right)}\dfrac{1}{d}\sum_{u\subseteq -i} \binom{d-1}{|u|}^{-1} \left(c(u\cup\lbrace i \rbrace) - c(u) \right),
\end{equation} with $c : \mathcal{P}(d) \longrightarrow \mathbb{R}$ a cost function which is specific to how input influence is measured. The $d$ input variables are interpreted as players (from the game theory framework from \cite{shapley1953value}) and the author of \cite{owen2014sobol} proposed to use as cost function the unormalized \textit{closed Sobol indices}, that are defined for all $u \in \mathcal{P}(d)$ by: \begin{equation}\label{closed_sobol}
    \text{VE}_u = \mathbb{V}\left[\mathbb{E}\left(\phi\left(\mathbf{X}\right)|\mathbf{X}_u\right)\right].
\end{equation}
The increment $\left(\text{VE}_{u\cup\lbrace i \rbrace} - \text{VE}_u\right)$ quantifies the individual contribution of the variable $X_i$ to the variance of the output in relation with the group of variables $u \subseteq -i$ taking into account both interaction and dependence. Moreover, the Shapley values using the alternative cost function $\text{EV}_u = \mathbb{E}\left[\mathbb{V}\left(\phi\left(\mathbf{X}\right)|\mathbf{X}_{-u}\right)\right]$ are equal to the ones using $\text{VE}_u$ \cite{song2016shapley}, which thus provides an alternative way to compute them. We call $\text{VE}_u$ and $\text{EV}_u$ the \textit{conditional indices}. Practical interest and theoretical properties of Shapley values for GSA have been widely studied since their introduction \cite{song2016shapley,owen2017shapley,iooss2019shapley}. Two important properties allow for an easy interpretation of these values: they are all non negative and sum to one. Thus, they give a quantitative measure, as a percentage, of the influence of each input on the variability of the output taking into account both interaction and dependence between input variables.

\begin{remark}\label{rem:val_closed_sobol}
Remark that $\text{VE}_\varnothing = \text{EV}_{[\![1,d]\!]} = 0$ and that $\text{VE}_{[\![1,d]\!]} = \text{EV}_\varnothing = \mathbb{V}\left(\phi\left(\mathbf{X}\right)\right)$. Thus, during the estimation process of the Shapley effects described in the remaining of the article, it will not be necessary to estimate the conditional indices for $u\in\lbrace\varnothing,[\![1,d]\!]\rbrace$. 
\end{remark}

\subsection{Shapley effect estimation schemes}

Obtaining an accurate estimation of the Shapley effects at a reasonable cost is very challenging and is an active research topic. In the context of game theory, the authors of \cite{castro2009polynomial} presented a first algorithm to estimate the Shapley effects which was improved by \cite{song2016shapley} in sensitivity analysis by reducing the number of calls to the function $\phi$. New approaches \cite{broto2020variance,plischke2021computing} and surrogate-model-based strategies \cite{iooss2019shapley,benoumechiara2019shapley,benard2022shaff} were explored to reduce even more the estimation cost of these indices while the articles \cite{broto2019block,broto2019sensitivity} were focused on the estimation of the Shapley effects with independent groups of variable.

The estimation schemes of the Shapley effects considered in this paper can be divided into two parts: \begin{enumerate}
    \item estimation of the conditional indices $\text{VE}_u$ or $\text{EV}_u$ for some subsets $u \in \mathcal{P}(d)\backslash\lbrace \varnothing, [\![1,d]\!]\rbrace$
    \item an aggregation procedure which consists in computing all the $\text{Sh}_i$ using the previous estimations of the conditional indices.
\end{enumerate}

In the following, two sampling methods from the literature are presented for the estimation of the conditional indices, with for each of them, an extension to the case where only an i.i.d. sample distributed according to the input distribution and its corresponding output is available. Afterwards, we also present two aggregation procedures.

\subsubsection{Estimation of $\text{EV}_u$ by double Monte Carlo}
\label{sss:2.2.1}

In this sub-subsection and the following one, for any $u \in \mathcal{P}(d)\backslash\lbrace \varnothing, [\![1,d]\!]\rbrace$, assume that: \begin{itemize}
    \item we can evaluate the code $\phi$ in any point of $\mathbb{X}$
    \item it is possible to generate an i.i.d. sample from the distribution of $\mathbf{X}_u$
    \item for any $\mathbf{x}_u \in \mathbb{X}_u$, it is possible to generate an i.i.d. sample from the distribution of $\mathbf{X}_{-u}|\mathbf{X}_u = \mathbf{x}_u$.
\end{itemize}

These assumptions define the \textit{given-model} framework. The two different cost functions $\text{VE}_u$ and $\text{EV}_u$ provide the same Shapley values, as mentioned above. However, the authors of \cite{sun2011efficient} pointed out a natural double (or two-level) Monte Carlo estimator of $\text{VE}_u$ but remarked that it is biased, whereas they suggested a natural double Monte Carlo estimator of $\text{EV}_u$ which is unbiased. Hence, the authors of \cite{song2016shapley} chose to estimate $\text{EV}_u$ instead of $\text{VE}_u$ and then suggested the following double Monte Carlo estimator:\begin{equation}\label{dmc}
    \widehat{\text{EV}}_{u,\text{MC}} = \dfrac{1}{N_u} \sum_{n=1}^{N_u} \dfrac{1}{N_I-1}\sum_{i=1}^{N_I} \left(\phi(\mathbf{X}_u^{(n,i)},\mathbf{X}_{-u}^{(n)}) - \overline{\phi(\mathbf{X}_{-u}^{(n)})} \right)^2,
\end{equation}
where $(\mathbf{X}_{-u}^{(n)})_{n \in [\![1,N_u]\!]}$ is an i.i.d. sample from the distribution of $\mathbf{X}_{-u}$, where for all $n \in [\![1,N_u]\!]$, $(\mathbf{X}_{u}^{(n,i)})_{i \in [\![1,N_I]\!]}$ is an i.i.d. sample from the distribution of $\mathbf{X}_u|\mathbf{X}_{-u} = \mathbf{X}_{-u}^{(n)}$ and where $\overline{\phi(\mathbf{X}_{-u}^{(n)})} = N_I^{-1}\sum\limits_{j=1}^{N_I} \phi(\mathbf{X}_u^{(n,j)},\mathbf{X}_{-u}^{(n)})$. This estimator is composed of an inner loop of size $N_I$ for the conditional variance, and of an outer loop of size $N_u$ (which depends on $u$) for the expectation. It requires $N_uN_I$ calls to $\phi$ and it is unbiased.

\subsubsection{Estimation of $\text{VE}_u$ by Pick-Freeze}

The basics of the \textit{Pick-Freeze} method were introduced in \cite{sobol1993sensitivity,homma1996importance}. When the components of $\mathbf{X}$ are independent, it is possible to remove the expensive double loop in \eqref{dmc} by rewriting the conditional indices $\text{VE}_u$ as a single expectation, with the interpretation of picking and freezing some input variables \cite{sobol1993sensitivity}. Recently, the Pick-Freeze method was generalized in \cite{broto2020variance} to the case where the inputs are correlated. The idea is to introduce a second random variable $\mathbf{X}^{u} = \left(\mathbf{X}_u,\mathbf{X}_{-u}'\right)$ with $\mathbf{X}_{-u}' \overset{d}{=}\mathbf{X}_{-u} | \mathbf{X}_u$ and $\mathbf{X}_{-u}'\perp\!\!\!\perp\mathbf{X}_{-u} | \mathbf{X}_u$, where $\perp\!\!\!\perp$ is the independence symbol, and to write:\begin{equation}\label{PF}\text{VE}_u = \mathbb{V}\left[\mathbb{E}\left(\phi(\mathbf{X}) | \mathbf{X}_u\right)\right] = \mathbb{E}\left[\phi(\mathbf{X})\phi(\mathbf{X}^u)\right] - \mathbb{E}\left[\phi(\mathbf{X})\right]^2.\end{equation} The random variables $\mathbf{X}$ and $\mathbf{X}^u$ are correlated and have the same distribution. To obtain $\mathbf{X}^u$ from $\mathbf{X}$, the component according to $u$ is frozen and the component according to $-u$ is chosen independently conditionally to $\mathbf{X}_u$.
In order to estimate the conditional index based on \eqref{PF}, let first $\widehat{E}_{\phi,N}$ be the natural Monte Carlo estimator of $\mathbb{E}\left[\phi(\mathbf{X})\right]$ with a sample of size $N$ from the distribution of $\mathbf{X}$. The following estimator of $\text{VE}_u$ was then suggested by \cite{broto2020variance}: \begin{equation}\label{pf}
    \widehat{\text{VE}}_{u,\text{PF}} = \frac{1}{N_u}\sum_{n=1}^{N_u}\phi(\mathbf{X}_u^{(n)},\mathbf{X}_{-u}^{(n,1)})\phi(\mathbf{X}_u^{(n)},\mathbf{X}_{-u}^{(n,2)}) - \left(\widehat{E}_{\phi,N}\right)^2,
\end{equation} where $ (\mathbf{X}_{u}^{(n)})_{n \in [\![1,N_u]\!]}$ is an i.i.d. sample from the distribution of $\mathbf{X}_u$ and where for all $n \in [\![1,N_u]\!]$, $(\mathbf{X}_{-u}^{(n,i)})_{i \in [\![1,2]\!]}$ are two independent random variables from the distribution of  $\mathbf{X}_{-u} | \mathbf{X}_{u} = \mathbf{X}_{u}^{(n)}$. The inner loop of size $N_I$ of the Monte Carlo estimator \eqref{dmc} is replaced by the product $\phi(\mathbf{X}_u^{(n)},\mathbf{X}_{-u}^{(n,1)})\phi(\mathbf{X}_u^{(n)},\mathbf{X}_{-u}^{(n,2)})$. This estimator requires $2N_u$ calls to $\phi$ and is unbiased.

\subsubsection{Extension when only an i.i.d. sample is available}
\label{sss:2.2.3}

In this section, assume that the code $\phi$ is no longer available and that only an i.i.d. sample $\left(\mathbf{X}^{(n)},\phi\left(\mathbf{X}^{(n)}\right)\right)_{n\in[\![1,N]\!]}$ with $\left(\mathbf{X}^{(n)}\right)_{n\in[\![1,N]\!]}$ from the distribution of $\mathbf{X}$ is available. This is the \textit{given-data} framework as defined in \cite{broto2020variance}. The estimation of sensitivity indices in this framework was first explored in \cite{da2013efficient} but in restrictive cases, for example when $|u| = 1$.

The authors of \cite{broto2020variance} extended the previous estimators \eqref{dmc} and \eqref{pf} to the given-data framework. The difficult point is that exact sampling from the conditional distributions $\mathbf{X}_u | \mathbf{X}_{-u} = \mathbf{x}_{-u}$ for some $\mathbf{x}_{-u} \in \mathbb{X}_{-u}$ is no longer possible. The \textit{nearest-neighbours approximation}, which is fully described in \cite{broto2020variance}, allows to approximate these distributions with the available i.i.d. sample. To that end, for $v \in \mathcal{P}(d)\backslash\lbrace\varnothing,[\![1,d]\!]\rbrace$ and $(l,i) \in [\![1,N]\!]^2$, let us write $k_N^v(l,i) \in [\![1,N]\!]$ for the index of the $i$-th nearest neighbour of the point $\mathbf{X}_v^{(l)}$ in the subspace $\mathbb{X}_v$ (according to the Euclidean distance) among $\left(\mathbf{X}_v^{(n)}\right)_{n \in [\![1,N]\!]}$. Moreover, $\left(k_N^v(l,i)\right)_{i \in [\![1,N]\!]}$ are defined to be two by two distinct: if several points are at equal distance from $\mathbf{X}_v^{(l)}$ for some $l \in [\![1,N]\!]$, ties are broken arbitrarily. 

Finally, the extended estimators of the conditional indices are given by: \begin{equation}\label{MC_KNN}
    \widehat{\text{EV}}_{u,\text{MC}}^{\text{KNN}} = \frac{1}{N_u}\sum_{n=1}^{N_u} \frac{1}{N_I-1}\sum_{i=1}^{N_I} \left[\phi\left(\mathbf{X}^{\left(k_N^{-u}(s(n),i)\right)}\right) - \frac{1}{N_I}\sum_{j=1}^{N_I}\phi\left(\mathbf{X}^{\left(k_N^{-u}(s(n),j)\right)}\right) \right]^2,
\end{equation} and 
\begin{equation}\label{PF_KNN}
     \widehat{\text{VE}}_{u,\text{PF}}^{\text{KNN}} = \frac{1}{N_u}\sum_{n=1}^{N_u}\phi\left(\mathbf{X}^{\left(k_N^{u}(s(n),1)\right)}\right)\phi\left(\mathbf{X}^{\left(k_N^{u}(s(n),2)\right)}\right) - \left(\widehat{E}_{\phi,N}\right)^2,
\end{equation}
with $\left(s(n)\right)_{n\in[\![1,N_u]\!]}$ a sample of uniformly distributed integers in $[\![1,N]\!]$. These estimators require no more calls to $\phi$ than those used to obtain the i.i.d. sample. The most costly step is the search of the nearest neighbours and under some assumptions, those given-data estimators are asymptotically consistent when $N$ and $N_u$ go to $+\infty$ \cite{broto2020variance}.

\subsubsection{Aggregation procedures}
\label{ss:aggreg}

The final part of the estimation of the Shapley effects is the \textit{aggregation procedure}. It consists in the use of the previous estimations of the conditional indices to deduce an estimation of the $d$ Shapley values.

The following procedure is immediate and natural, and is called \textit{subset procedure} in \cite{broto2020variance}: \begin{enumerate}
    \item estimate the conditional indices for all $u \in \mathcal{P}(d)\backslash\lbrace\varnothing,[\![1,d]\!]\rbrace$ (see \Cref{rem:val_closed_sobol})
    \item for all $i\in[\![1,d]\!]$, estimate $\text{Sh}_i$ with \eqref{shapley}.
\end{enumerate} 
In the given-model framework, the computational cost to  estimate all the Shapley values with this aggregation procedure is then $N_{\mathbb{V}} + (2^d-2)N_IN_O$ with the double Monte Carlo method and $N_{\mathbb{V}} + 2(2^d-2)N_O$ with the Pick-Freeze method, where $N_{\mathbb{V}}$ is the size of the sample used to estimate $\mathbb{V}\left(\phi\left(\mathbf{X}\right)\right)$ and $N_O = N_u$ for all $u \in \mathcal{P}(d)\backslash\lbrace \varnothing, [\![1,d]\!]\rbrace$ (notation of \eqref{dmc}). Note that $2^d-2$ is the cardinal of $\mathcal{P}(d)\backslash\lbrace\varnothing,[\![1,d]\!]\rbrace$. Its computational cost prohibits its direct use when $d$ increases, however, note that \cite{broto2020variance} also suggests using different values of $N_u$ for $u\in\mathcal{P}(d)\backslash\lbrace\varnothing,[\![1,d]\!]\rbrace$ to tackle larger values of $d$ with a bounded computational budget.

The \textit{random-permutation procedure} was introduced by \cite{castro2009polynomial} in the context of game theory and its computational cost was later reduced by \cite{song2016shapley}. The idea is to rewrite the Shapley effect $\text{Sh}_i$ as an expectation over the set of all the permutations of $[\![1,d]\!]$, denoted as $\mathcal{S}(d)$: \begin{equation}
    \forall i \in [\![1,d]\!],\ \text{Sh}_i = \dfrac{1}{\mathbb{V}\left(\phi\left(\mathbf{X}\right)\right)} \mathbb{E}_\Pi\left(\text{VE}_{P_i(\Pi) \cup \lbrace i \rbrace} - \text{VE}_{P_i(\Pi)}\right),
\end{equation} where for $\pi \in \mathcal{S}(d), \ P_i(\pi) = \left\lbrace \pi(j) / j \in [\![1,\pi^{-1}(i)-1]\!]\right\rbrace$ and $\Pi$ is a random variable uniformly distributed over $\mathcal{S}(d)$. The expectation is then estimated using an i.i.d. sample $(\pi_j)_{j\in[\![1,m]\!]}$ of permutations uniformly distributed over $\mathcal{S}(d)$ with $m \ll d!$. In the given-model framework, the computational cost of the improved algorithm to estimate all the Shapley values proposed by \cite{song2016shapley} is $N_{\mathbb{V}} + m(d-1)N_IN_O$ with the double Monte Carlo method and $N_{\mathbb{V}} + 2m(d-1)N_O$ with the Pick-Freeze method. The numerical experiments in \cite{broto2020variance} suggest that it has a higher variance than the subset procedure but its computational cost can be controlled with the parameter $m$.

\subsection{Extension to reliability-oriented sensitivity analysis}

Reliability analysis consists in the estimation of the failure probability $p_t = \mathbb{P}\left(\phi\left(\mathbf{X}\right)>t\right)$, for a fixed known threshold $t\in\mathbb{R}$. Classical Monte Carlo sampling is not adapted to this problem when $p_t$ is getting smaller because its computational cost becomes too large to obtain an accurate estimation. Therefore, several techniques have been developed in order to estimate $p_t$ more accurately: one can mentioned FORM/SORM methods \cite{hasofer1974exact,breitung1984asymptotic}, subset sampling \cite{cerou:inria-00584352} or line sampling \cite{koutsourelakis2004reliability} for example. Another method, importance sampling, is reviewed here before coming back to sensitivity analysis for reliability purpose.

\subsubsection{Importance sampling}
\label{ss:is}
\textit{Importance sampling} (IS) is a very usual variance-reduction technique which was introduced by \cite{Kahn1951SplittingParticleTransmission} and applied for the first time in reliability analysis by \cite{Shinozuka1983BasicAO}. In the case of the estimation of a failure probability $p_t = \mathbb{P}\left(\phi\left(\mathbf{X}\right)>t\right) = \mathbb{E}_{f_{\mathbf{X}}}\left(\mathbf{1}\left(\phi\left(\mathbf{X}\right)>t\right)\right)$, it consists in rewriting the expectation according to an auxiliary density $g : \mathbb{X} \longrightarrow \mathbb{R}_+$  as $\mathbb{E}_{g}\left(\mathbf{1}\left(\phi\left(\mathbf{X}\right)>t\right)w^g\left(\mathbf{X}\right)\right)$, where $w^g\left(\mathbf{x}\right) = f_{\mathbf{X}}(\mathbf{x})/g(\mathbf{x})$ is the \textit{likelihood ratio}. To get an unbiased estimate, the support of $g$ must contain the support of $\mathbf{x}\in\mathbb{X} \mapsto \mathbf{1}\left(\phi\left(\mathbf{x}\right)>t\right)f_{\mathbf{X}}(\mathbf{x)}$. The corresponding estimator is then given by: \begin{equation}\label{ptisest}\widehat{p}_{t,N}^{IS} = \frac{1}{N}\sum_{n=1}^N \mathbf{1}\left(\phi(\mathbf{X}^{(n)}) > t\right)w^g\left(\mathbf{X}^{(n)}\right),\end{equation} with $\left(\mathbf{X}^{(n)}\right)_{n\in [\![1,N]\!]}$ an i.i.d. sample distributed according to the IS auxiliary distribution $g$. It is consistent and unbiased, and it has zero-variance if and only if $g = g_{\text{opt}}$ with $\forall \mathbf{x} \in \mathbb{X}$, $g_{\text{opt}}\left(\mathbf{x}\right) \propto \mathbf{1}\left(\phi\left(\mathbf{x}\right)>t\right)f_{\mathbf{X}}(\mathbf{x)}$, which is the density of the input distribution with PDF $f_{\mathbf{X}}$ restricted to the failure domain \cite{bucklew2004introduction}. This optimal density can not be considered in practice because the normalizing constant is $p_t$, which is the quantity to estimate, but many techniques exist to approach $g_{\text{opt}}$ by a near-optimal auxiliary density: methods based on the design point \cite{harbitz1983efficient}, non-parametric methods \cite{zhang1996nonparametric} or the cross-entropy method \cite{de2005tutorial,rubinstein2013cross}.

\subsubsection{Target sensitivity analysis}

Reliability-oriented sensitivity analysis can be divided into two categories, regarding the goal considered \cite{raguet2018target,MARREL2021107711}: \textit{target sensitivity analysis} and \textit{conditional sensitivity analysis}. In this article, only the first one will be examined. TSA combines both reliability and sensitivity analyses and aims at studying the influence of each input variable on the occurrence of the failure event. To that end, one can apply the previous variance-based global sensitivity indices described in \Cref{subsec:2.1} to the quantity of interest $\mathbf{1}\left(\phi\left(\mathbf{X}\right) > t\right)$ instead of the output $Y = \phi\left(\mathbf{X}\right)$: these new indices are called \textit{target indices} \cite{MARREL2021107711,ILIDRISSI2021105115}. In particular, the definitions \eqref{sobol} to \eqref{closed_sobol} are directly extended by replacing $\phi\left(\mathbf{X}\right)$ by $\mathbf{1}\left(\phi\left(\mathbf{X}\right) > t\right)$. Techniques based on Monte Carlo sampling \cite{wei2012efficient} and non-parametric estimation \cite{perrin2019efficient} have been proposed to estimate first-order and total-order target Sobol indices, whereas as far as we know, estimators of the target Shapley effects $\text{T-Sh}_i$ have only been recently suggested in \cite{ILIDRISSI2021105115} and are strongly inspired from those in \cite{broto2020variance}. The estimators from \cite{ILIDRISSI2021105115} are based on a classic Monte Carlo sampling from the input distribution of PDF $f_{\mathbf{X}}$ in both given-model and given-data frameworks. In the latter framework in \cite{ILIDRISSI2021105115}, a reliability analysis is first performed and leads to the estimation of $p_t$ with a Monte Carlo sampling. Then, the target Shapley effects are estimated from the available sample with the estimator in \eqref{MC_KNN} combined with the random permutation procedure. However, numerical experiments in \Cref{sec:numerical_experiments} highlight their limits when the failure probability is getting smaller: their variance increases and a large Monte Carlo sample is thus required to get an accurate estimation, and make them hardly applicable with a costly computer model $\phi$. The main goal of this article is thus to remedy these shortcomings with importance sampling.

\section{Target Shapley effects estimation by importance sampling}
\label{sec:contribution}

In this section, we suggest new estimators of the $d$ target Shapley effects $\text{T-Sh}_i$, defined as in \eqref{shapley} with $\phi\left(\mathbf{X}\right)$ replaced by $\mathbf{1}\left(\phi\left(\mathbf{X}\right)>t\right)$, by importance sampling when the input variables are correlated. To that end, we propose four new estimators of the target conditional indices $\text{T-VE}_u$ and $\text{T-EV}_u$, defined as $\text{VE}_u$ and $\text{EV}_u$ with $\phi\left(\mathbf{X}\right)$ replaced by $\mathbf{1}\left(\phi\left(\mathbf{X}\right)>t\right)$, based on those of the previous section: \begin{enumerate}
    \item an unbiased estimator of $\text{T-EV}_u$ by double Monte Carlo with importance sampling given-model
    \item an estimator of $\text{T-EV}_u$ by double Monte Carlo with importance sampling given-data
    \item an unbiased estimator of $\text{T-VE}_u$ by Pick-Freeze with importance sampling given-model
    \item an estimator of $\text{T-VE}_u$ by Pick-Freeze with importance sampling given-data.
\end{enumerate}
Recall that the given-model framework is described at the beginning of \Cref{sss:2.2.1} and that the given-data framework is described in \Cref{sss:2.2.3}. Then, the two aggregation procedures described in \Cref{ss:aggreg} provide eight new estimators of the target Shapley effects which have a lower variance than the existing ones when the auxiliary density is adapted to the problem, as will be shown numerically in \Cref{sec:numerical_experiments}.

For any subset $u \in \mathcal{P}(d)\backslash\lbrace \varnothing, [\![1,d]\!]\rbrace$, for any $\mathbf{x}_u \in \mathbb{X}_u$, we let $f_{\mathbf{X}_u}$ and $f_{\mathbf{X}_{-u}|\mathbf{X}_u = \mathbf{x}_u}$ denote respectively the PDF of the marginal distribution of $\mathbf{X}_u$ and the PDF of the conditional distribution of $\mathbf{X}_{-u}|\mathbf{X}_u = \mathbf{x}_u$. In addition, let $g:\mathbb{X}\to\mathbb{R}_+$ be the PDF of the IS auxiliary distribution as in \Cref{ss:is} and $g_{\mathbf{X}_u}$ and $g_{\mathbf{X}_{-u}|\mathbf{X}_u = \mathbf{x}_u}$ be the PDFs of its marginal and its conditional distributions. Moreover, in order to lighten the notations, for all $t\in\mathbb{R}$ and for all $\mathbf{x} \in \mathbb{X}$, let us write: \begin{equation}\psi_t\left(\mathbf{x}\right) = \mathbf{1}\left(\phi\left(\mathbf{x}\right)>t\right) \mbox{ and } w_t^g\left(\mathbf{x}\right) = \psi_t\left(\mathbf{x}\right)\dfrac{f_{\mathbf{X}}\left(\mathbf{x}\right)}{g\left(\mathbf{x}\right)}.\end{equation} Finally, in the following, the convention $0/0=0$ will be used, and in particular, this implies that for all $\mathbf{x}_{-u} \in \mathbb{X}_{-u}$, $f_{\mathbf{X}_{-u}|\mathbf{X}_u = \mathbf{x}_u}\left(\mathbf{x}_{-u}\right) = 0$ for any $\mathbf{x}_u \in \mathbb{X}_u$ such that $f_{\mathbf{X}_u}\left(\mathbf{x}_u\right) = 0$ (see also \Cref{remark:3.1,remark:3.2}).

\subsection{Estimation of $\text{T-EV}_u$ by double Monte Carlo with importance sampling}
\label{sec:3.1}

For the same reasons as in \cite{song2016shapley}, for $u \in \mathcal{P}(d)\backslash\lbrace \varnothing, [\![1,d]\!]\rbrace$, we choose to estimate the target conditional index $\text{T-EV}_u$ by double Monte Carlo. However, contrary to the existing literature, in order to introduce importance sampling in the estimation process we write this target conditional index as: \begin{equation}\label{lemma4.1}
    \text{T-EV}_u = \mathbb{E}_{f_{\mathbf{X}}}\left[\mathbb{V}_{f_{\mathbf{X}}}\left(\psi_t(\mathbf{X})|\mathbf{X}_{-u} \right) \right] = p_t - \mathbb{E}_{f_{\mathbf{X}}}\left[\mathbb{E}_{f_{\mathbf{X}}}\left(\psi_t(\mathbf{X})|\mathbf{X}_{-u}\right)^2 \right].
\end{equation} The estimator $\widehat{p}_{t,N}^{\text{IS}}$ in \eqref{ptisest} already provides an unbiased and convergent estimation of the failure probability $p_t$ by importance sampling. The main problem is thus to estimate $\mathbb{E}_{f_{\mathbf{X}}}\left[\mathbb{E}_{f_{\mathbf{X}}}\left(\psi_t(\mathbf{X})|\mathbf{X}_{-u}\right)^2 \right]$ by importance sampling. Let us rewrite this quantity according to the IS auxiliary density $g$ (see \ref{app:proof_lemma_fund_is_dmc}): \begin{equation}\label{ismc1}
    \mathbb{E}_{f_{\mathbf{X}}}\left[\mathbb{E}_{f_{\mathbf{X}}}\left(\psi_t(\mathbf{X})|\mathbf{X}_{-u}\right)^2 \right] = \mathbb{E}_g\left[\mathbb{E}_g\left(\psi_t(\mathbf{X})\left.\dfrac{f_{\mathbf{X}_u|\mathbf{X}_{-u}}\left(\mathbf{X}_u\right)}{g_{\mathbf{X}_u|\mathbf{X}_{-u}}\left(\mathbf{X}_u\right)}\right|\mathbf{X}_{-u}\right)^2 \dfrac{f_{\mathbf{X}_{-u}}\left(\mathbf{X}_{-u}\right)}{g_{\mathbf{X}_{-u}}\left(\mathbf{X}_{-u}\right)}\right].
\end{equation} However, in the rest of the article, we will not assume that it is possible to evaluate directly the conditional PDFs $f_{\mathbf{X}_u|\mathbf{X}_{-u} = \mathbf{x}_{-u}}$ and $g_{\mathbf{X}_u|\mathbf{X}_{-u} = \mathbf{x}_{-u}}$ at any point of $\mathbb{X}_u$ for all $\mathbf{x}_{-u} \in \mathbb{X}_{-u}$, as it can be a restrictive hypothesis in practice, especially with non trivial input distributions. Therefore, using the definition of the conditional PDF, we choose to rewrite the ratio $f_{\mathbf{X}_u|\mathbf{X}_{-u}}/g_{\mathbf{X}_u|\mathbf{X}_{-u}}$ in the form $f_{\mathbf{X}}/f_{\mathbf{X}_{-u}} \times g_{\mathbf{X}_{-u}}/g_{\mathbf{X}}$. From this point, some calculations developed in \ref{app:proof_lemma_fund_is_dmc} lead to the following lemma. \begin{lemma}\label{fund_is_dmc}For any IS auxiliary density $g : \mathbb{X}\longrightarrow\mathbb{R}_+$, we have:
\begin{equation}\label{ismc}
    \mathbb{E}_{f_{\mathbf{X}}}\left[\mathbb{E}_{f_{\mathbf{X}}}\left(\psi_t(\mathbf{X})|\mathbf{X}_{-u}\right)^2 \right] = \mathbb{E}_g\left[\mathbb{E}_g\left(w_t^g(\mathbf{X})\left.\right|\mathbf{X}_{-u}\right)^2 \dfrac{g_{\mathbf{X}_{-u}}\left(\mathbf{X}_{-u}\right)}{f_{\mathbf{X}_{-u}}\left(\mathbf{X}_{-u}\right)}\right].
\end{equation}
\end{lemma} \begin{proof}See \ref{app:proof_lemma_fund_is_dmc}.\end{proof} Because of the described transformation, one can notice that the likelihood ratios in \eqref{ismc} are not all in the classic form $f_{\mathbf{X}}/g$, there is a term in the form $g/f_{\mathbf{X}}$ in the outer expectation. From \eqref{ismc}, it is thus possible to introduce double Monte Carlo estimators by importance sampling of the target conditional index $\text{T-EV}_u$ in both given-model and given-data frameworks.

\subsubsection{Given-model framework with importance sampling}
\label{sss:MC_ngd}
In the given-model framework with importance sampling, assume that:\begin{itemize}
    \item we can evaluate the code $\phi$ in any point of $\mathbb{X}$
    \item it is possible to generate an i.i.d. sample from the distribution of $g_{\mathbf{X}_{-u}}$
    \item for any $\mathbf{x}_{-u} \in \mathbb{X}_{-u}$, it is possible to generate an i.i.d. sample from the distribution of $g_{\mathbf{X}_{u}|\mathbf{X}_{-u} = \mathbf{x}_{-u}}$
    \item it is possible to evaluate $f_{\mathbf{X}}$ and $g$ in any point of $\mathbb{X}$
    \item it is possible to evaluate $f_{\mathbf{X}_{-u}}$ and $g_{\mathbf{X}_{-u}}$ in any point of $\mathbb{X}_{-u}$.
\end{itemize} The third hypothesis is reasonable because we are free to choose $g$ such that the sampling from the conditional distributions is not problematic. Moreover, it is no longer necessary to evaluate the conditional PDFs of $f_{\mathbf{X}}$ and $g$, which is convenient as discussed above. The new writing \eqref{ismc} suggests then to introduce the following double Monte Carlo estimator by importance sampling: \begin{equation}\label{tevngd}
    \widehat{\text{T-EV}}_{u,\text{MC}}^{\text{IS}} = \widehat{p}_{t,N}^{\text{IS}} - \left( \dfrac{1}{N_u}\sum_{n=1}^{N_u}\left(\overline{\psi_t^{\text{IS}}(\mathbf{X}_{-u}^{(n)})}\right)^2\dfrac{g_{\mathbf{X}_{-u}}\left(\mathbf{X}_{-u}^{(n)}\right)}{f_{\mathbf{X}_{-u}}\left(\mathbf{X}_{-u}^{(n)}\right)} - \widehat{E}_{\text{bias},u}^{IS} \right),
\end{equation} where $\left(\mathbf{X}_{-u}^{(n)}\right)_{n \in [\![1,N_u]\!]}$ is an i.i.d. sample from the distribution of $g_{\mathbf{X}_{-u}}$, where for all $n \in [\![1,N_u]\!]$, $\left(\mathbf{X}_{u}^{(n,k)}\right)_{k \in [\![1,N_I]\!]}$ is an i.i.d. sample from the distribution of $g_{\mathbf{X}_u|\mathbf{X}_{-u} = \mathbf{X}_{-u}^{(n)}}$ and \begin{equation}\label{inner_loop}
    \overline{\psi_t^{\text{IS}}\left(\mathbf{X}_{-u}^{(n)}\right)} = \dfrac{1}{N_I} \sum\limits_{k=1}^{N_I}w_t^g\left(\mathbf{X}_u^{(n,k)},\mathbf{X}_{-u}^{(n)}\right),
\end{equation} and where the term $\widehat{E}_{\text{bias},u}^{\text{IS}}$ is defined by: \begin{equation}\label{bias_est}\widehat{E}_{\text{bias},u}^{\text{IS}} = \dfrac{1}{N_u} \sum\limits_{n=1}^{N_u} \dfrac{1}{N_I-1}\left[\dfrac{1}{N_I} \sum\limits_{k=1}^{N_I}w_t^g\left( \mathbf{X}_u^{(n,k)},\mathbf{X}_{-u}^{(n)}\right)^2 - \left(\overline{\psi_t^{\text{IS}}\left(\mathbf{X}_{-u}^{(n)}\right)}\right)^2\right] \dfrac{g_{\mathbf{X}_{-u}}\left(\mathbf{X}_{-u}^{(n)}\right)}{f_{\mathbf{X}_{-u}}\left(\mathbf{X}_{-u}^{(n)}\right)}.\end{equation}
As in \cite{song2016shapley}, this estimator is composed of an inner loop of size $N_I$ for the inner conditional expectation and an outer loop of size $N_u$ (which depends on $u$) for the outer expectation. Moreover, it is clear that for all $n \in [\![1,N_u]\!]$, $\overline{\psi_t^{\text{IS}}(\mathbf{X}_{-u}^{(n)})}$ is an unbiased estimator of $\mathbb{E}_g\left(w_t^g(\mathbf{X})\left.\right|\mathbf{X}_{-u} = \mathbf{X}_{-u}^{(n)}\right)$. However, $\left(\overline{\psi_t^{\text{IS}}(\mathbf{X}_{-u}^{(n)})}\right)^2$ is not an unbiased estimator of $\mathbb{E}_g\left(w_t^g(\mathbf{X})\left.\right|\mathbf{X}_{-u} = \mathbf{X}_{-u}^{(n)}\right)^2$ because of the square which creates a bias. This bias from the inner loop spreads itself into the outer loop and finally the bias of the uncorrected double Monte Carlo estimator by importance sampling of $\text{T-EV}_u$ (obtained by removing $\widehat{E}_{\text{bias},u}^{IS}$ in \eqref{tevngd}) is $\mathbb{E}_g\left[\mathbb{V}_g\left(\overline{\psi_t^{\text{IS}}(\mathbf{X}_{-u})}|\mathbf{X}_{-u}\right)g_{\mathbf{X}_{-u}}(\mathbf{X}_{-u})\left/f_{\mathbf{X}_{-u}}(\mathbf{X}_{-u})\right.\right]$, where for all $\mathbf{x}_{-u} \in \mathbb{X}_{-u}$: \begin{equation}\label{psi_t_is}
    \overline{\psi_t^{\text{IS}}\left(\mathbf{x}_{-u}\right)} = \dfrac{1}{N_I} \sum\limits_{k=1}^{N_I}w_t^g\left(\mathbf{X}_u^{(k)},\mathbf{x}_{-u}\right),
\end{equation} with $\left(\mathbf{X}_u^{(k)}\right)_{k\in[\![1,N_I]\!]}$ an i.i.d. sample from the distribution of $g_{\mathbf{X}_u|\mathbf{X}_{-u} = \mathbf{x}_{-u}}$. The term $\widehat{E}_{\text{bias},u}^{IS}$ in \eqref{bias_est} is then an unbiased estimator of the latter bias and allows to correct it. This leads to the main result of this section: \begin{proposition}\label{dmc_is_ub}$\widehat{\text{T-EV}}_{u,\text{MC}}^{\text{IS}}$ in \eqref{tevngd} is an unbiased double Monte Carlo estimator by importance sampling of the conditional index $\text{T-EV}_u$, and requires $N_uN_I$ calls to the function $\phi$ in addition to those to estimate $p_t$.\end{proposition}\begin{proof}See \ref{app:proof_prop_dmc_is_ub}.\end{proof}

\subsubsection{Given-data framework with importance sampling}
\label{sss:MC_gd}
In the given-data framework with importance sampling, assume that: \begin{itemize}
    \item an i.i.d. input/output sample $\left(\mathbf{X}^{(n)},\psi_t\left(\mathbf{X}^{(n)}\right)\right)_{n\in[\![1,N]\!]}$ with $\left(\mathbf{X}^{(n)}\right)_{n\in[\![1,N]\!]}$ distributed according to the IS auxiliary distribution $g$ is available
    \item it is possible to evaluate $f_{\mathbf{X}}$ and $g$ in any point of $\mathbb{X}$
    \item it is possible to evaluate $f_{\mathbf{X}_{-u}}$ and $g_{\mathbf{X}_{-u}}$ in any point of $\mathbb{X}_{-u}$.
\end{itemize}
Contrary to the given-model framework, here the black-box model $\phi$ is no longer available and it is not possible to generate any additional sample from any distribution. Typically, in the context of ROSA, the i.i.d. sample from $g$ is already obtained from a previous importance-sampling-based estimation of $p_t$ (see \Cref{ss:3.3}). The extension of the double Monte Carlo estimator of $\text{T-EV}_u$ by importance sampling \eqref{tevngd} to the given-data framework is based on the nearest-neighbours method introduced in \cite{broto2020variance} and briefly described in \Cref{sec:review}. Here, for $i\in[\![1,N]\!]$, the exact sampling from the conditional auxiliary PDF $g_{\mathbf{X}_u|\mathbf{X}_{-u} = \mathbf{X}_{-u}^{(i)}}$ is approximated by the $N_I$ nearest neighbours of $\mathbf{X}_{-u}^{(i)}$ among the sample $\left(\mathbf{X}_{-u}^{(n)}\right)_{n\in[\![1,N]\!]}$. Then, the extended estimator is: \begin{equation}\label{tevgd}
    \widehat{\text{T-EV}}_{u,\text{MC}}^{\text{IS},\text{KNN}} = \widehat{p}_{t,N}^{\text{IS}} - \left( \dfrac{1}{N_u}\sum_{n=1}^{N_u}\left(\overline{\psi_{t,u}^{\text{IS},\text{KNN}}(\mathbf{X}^{(s(n))})}\right)^2\dfrac{g_{\mathbf{X}_{-u}}\left(\mathbf{X}_{-u}^{(s(n))}\right)}{f_{\mathbf{X}_{-u}}\left(\mathbf{X}_{-u}^{(s(n))}\right)} - \widehat{E}_{\text{bias},u}^{\text{IS},\text{KNN}} \right),
\end{equation} where $\left(s(n)\right)_{n\in[\![1,N_u]\!]}$ is a sample of uniformly distributed integers in $[\![1,N]\!]$, where for all $n\in[\![1,N_u]\!]$, \begin{equation}
    \overline{\psi_{t,u}^{\text{IS},\text{KNN}}\left(\mathbf{X}^{(s(n))}\right)} = \dfrac{1}{N_I} \sum\limits_{k=1}^{N_I}w_t^g\left(\mathbf{X}^{(k_N^{-u}(s(n),k))}\right),
\end{equation} with $k_N^{-u}$ defined in \Cref{sss:2.2.3} and where the term $\widehat{E}_{\text{bias},u}^{\text{IS},\text{KNN}}$ is defined by: \begin{equation}\label{bias_est_gd}\widehat{E}_{\text{bias},u}^{\text{IS},\text{KNN}} = \dfrac{1}{N_u} \sum\limits_{n=1}^{N_u} \dfrac{1}{N_I-1}\left[\dfrac{1}{N_I} \sum\limits_{k=1}^{N_I}w_t^g\left(\mathbf{X}^{(k_N^{-u}(s(n),k))}\right)^2 - \left(\overline{\psi_{t,u}^{\text{IS},\text{KNN}}\left(\mathbf{X}^{(s(n))}\right)}\right)^2\right]\dfrac{g_{\mathbf{X}_{-u}}\left(\mathbf{X}_{-u}^{(s(n))}\right)}{f_{\mathbf{X}_{-u}}\left(\mathbf{X}_{-u}^{(s(n))}\right)}.\end{equation}
It does not require any additional call to $\phi$ to those from the already available i.i.d. sample. In the same way as the existing given-data estimators in GSA and ROSA (see \Cref{sec:review}), the costly part of this algorithm is the computation of the nearest neighbours at each step. Moreover, both estimators \eqref{tevngd} and \eqref{tevgd} have the same general structure. In \eqref{tevngd}, the bias created by the square in the inner loop is corrected by the estimator \eqref{bias_est}. This structure is kept in \eqref{tevgd} and \eqref{bias_est_gd}. Nevertheless, for given finite values of $N_u$ and $N_I$, the given-data estimator \eqref{tevgd} still has a bias caused by the nearest neighbour approximation, and as far as we know, no article has proposed any theoretical study of this bias.

\begin{remark}
Let us provide some additional motivations for the given-data with importance sampling framework. First, we assumed here that the computer model is no longer available and that we can only use the available sample to estimate the target Shapley effects. One can notice that it might be possible to get around this problem by fitting a surrogate model, such as a Gaussian process for example, with the available sample and then coming back to the given-model framework. However, this approach has some drawbacks that a practitioner may prefer to avoid. In fact, when $N$ is large, it is not straightforward to fit efficiently a surrogate model, even more when the input dimension increases. Furthermore, even if the given-model framework with a surrogate model does not suffer from the nearest-neighbour approximation, the surrogate model introduces another approximation error which is hard to quantify and is added to the estimation error of each index. This error is not necessarily lower than the error due to the nearest-neighbour approximation when the dimension increases.

Second, we assumed that it is not possible to draw additional samples from any distribution. In the given-model framework, it is necessary to be able to draw samples according to the conditional distributions of $g$. In practice, most of the time, Gaussian auxiliary distributions are chosen and it is then straightforward to sample from their conditional distributions. However, in some cases, auxiliary distributions belonging to other parametric families could be more relevant but do not satisfy this criterion.\end{remark}

\subsection{Estimation of $\text{T-VE}_u$ by Pick-Freeze with importance sampling}

For any subset $u \in \mathcal{P}(d)\backslash\lbrace \varnothing, [\![1,d]\!]\rbrace$, we estimate the target conditional index $\text{T-VE}_u$ by Pick-Freeze and with importance sampling. In ROSA with correlated inputs, we recall the fundamental equation of Pick-Freeze \eqref{PF}: \begin{equation}\label{PF_fiab_dep}
    \text{T-VE}_u = \mathbb{V}_{f_{\mathbf{X}}}\left[\mathbb{E}_{f_{\mathbf{X}}}\left(\psi_t(\mathbf{X}) | \mathbf{X}_u\right)\right] = \mathbb{E}_{f_{\mathbf{X}}}\left[\psi_t(\mathbf{X})\psi_t(\mathbf{X}^u)\right] - p_t^2,
\end{equation} where $\mathbf{X}^{u} = \left(\mathbf{X}_u,\mathbf{X}_{-u}'\right)$ with $\mathbf{X}_{-u}' \overset{d}{=}\mathbf{X}_{-u} | \mathbf{X}_u$ and $\mathbf{X}_{-u}'\perp\!\!\!\perp\mathbf{X}_{-u} | \mathbf{X}_u$. For $\left(\mathbf{X}^{(n)}\right)_{n\in[\![1,N]\!]}$ an i.i.d. sample distributed according to the IS auxiliary distribution $g$, the estimator \begin{equation}\label{pt2_is_ub}\widehat{p}_{t,N}^{\text{IS},\text{ub}} = \left(\widehat{p}_{t,N}^{\text{IS}}\right)^2 - \dfrac{1}{N-1}\left[\dfrac{1}{N} \sum\limits_{n=1}^{N} w_t^g\left(\mathbf{X}^{(n)}\right)^2 - \left(\widehat{p}_{t,N}^{\text{IS}}\right)^2\right]\end{equation} is easily shown (see \ref{app:ub_pt2}) to be an unbiased estimator by importance sampling of $p_t^2$. The main problem is then to estimate the expectation $\mathbb{E}_{f_{\mathbf{X}}}\left[\psi_t(\mathbf{X})\psi_t(\mathbf{X}^u)\right]$ by importance sampling. Let us rewrite it according to the IS auxiliary density $g$:
\begin{lemma}\label{fund_is_pf}For any IS auxiliary density $g:\mathbb{X}\longrightarrow\mathbb{R}_+$, we have:\begin{equation}\label{is_pf}
    \mathbb{E}_{f_{\mathbf{X}}}\left(\psi_t(\mathbf{X})\psi_t(\mathbf{X}^u) \right) = \mathbb{E}_g\left(w_t^g\left(\mathbf{X} \right)w_t^g\left(\mathbf{X}^u \right)\dfrac{g_{\mathbf{X}_u}(\mathbf{X}_u)}{f_{\mathbf{X}_u}(\mathbf{X}_u) }\right).
\end{equation}\end{lemma} \begin{proof}See \ref{app:proof_lemma_fund_is_pf}.\end{proof}The term in the expectation is a function of the three random variables $\mathbf{X}_u$, $\mathbf{X}_{-u}$ and $\mathbf{X}_{-u}'$ with the correlation structure described above. From \eqref{is_pf}, it is then possible to propose Pick-Freeze estimators by importance sampling of the target conditional index $\text{T-VE}_u$ in both given-model and given-data frameworks.

\subsubsection{Given-model framework with importance sampling}

Here, we make the assumptions of the given-model with importance sampling framework defined in \Cref{sss:MC_ngd}. The writing \eqref{is_pf} suggests then to introduce the following Pick-Freeze estimator of $\text{T-VE}_u$ by importance sampling: \begin{equation}\label{tvengd}\widehat{\text{T-VE}}_{u,\text{PF}}^{\text{IS}} = \dfrac{1}{N_u}\sum_{n=1}^{N_u}w_t^g\left(\mathbf{X}_u^{(n)},\mathbf{X}_{-u}^{(n,1)}\right)w_t^g\left(\mathbf{X}_u^{(n)},\mathbf{X}_{-u}^{(n,2)}\right)\dfrac{g_{\mathbf{X}_u}(\mathbf{X}_u^{(n)})}{f_{\mathbf{X}_u}(\mathbf{X}_u^{(n)})} - \widehat{p}_{t,N}^{\text{IS},\text{ub}} ,\end{equation} where $(\mathbf{X}_{u}^{(n)})_{n \in [\![1,N_u]\!]}$ is an i.i.d. sample from the distribution of $g_{\mathbf{X}_u}$ and where for all $n \in [\![1,N_u]\!]$, $(\mathbf{X}_{-u}^{(n,k)})_{k \in [\![1,2]\!]}$ are two independent random variables from the distribution of  $g_{\mathbf{X}_{-u} | \mathbf{X}_{u} = \mathbf{X}_{u}^{(n)}}$. This then leads to the main result of this sub-section:\begin{proposition}\label{pf_is_ub}
$\widehat{\text{T-VE}}_{u,\text{PF}}^{\text{IS}}$ in \eqref{tvengd} is an unbiased estimator of $\text{T-VE}_u$ and it requires $2N_u$ calls to the function $\phi$ in addition to those to estimate $p_t^2$.
\end{proposition}
\begin{proof}
The first term in the estimator \eqref{tvengd} is an unbiased estimator of the expectation $\mathbb{E}_g\left(w_t^g\left(\mathbf{X} \right)w_t^g\left(\mathbf{X}^u \right)\dfrac{g_{\mathbf{X}_u}(\mathbf{X}_u)}{f_{\mathbf{X}_u}(\mathbf{X}_u) }\right)$ since it is its empirical mean, and $\widehat{p}_{t,N}^{\text{IS},\text{ub}}$ in \eqref{pt2_is_ub} is an unbiased estimator of $p_t^2$. The linearity of the expectation concludes the proof. \end{proof}

\subsubsection{Given-data framework with importance sampling}

We make here the same assumptions of the given-data with importance sampling framework defined in \Cref{sss:MC_gd}. In particular, an i.i.d. input/output sample $\left(\mathbf{X}^{(n)},\psi_t\left(\mathbf{X}^{(n)}\right)\right)_{n\in[\![1,N]\!]}$ with $\left(\mathbf{X}^{(n)}\right)_{n\in[\![1,N]\!]}$ distributed according to the IS auxiliary distribution $g$ is available. Once more, we will use the nearest-neighbour approximation to extend the Pick-Freeze estimator by importance sampling \eqref{tvengd} of $\text{T-VE}_u$ to the given-data framework and the corresponding estimator is:\begin{equation}\label{tvegd}\widehat{\text{T-VE}}_{u,\text{PF}}^{\text{IS},\text{KNN}} = \dfrac{1}{N_u}\sum_{n=1}^{N_u}w_t^g\left(\mathbf{X}^{(k_N^{u}(s(n),1))}\right)w_t^g\left(\mathbf{X}^{(k_N^{u}(s(n),2))}\right)\dfrac{g_{\mathbf{X}_u}(\mathbf{X}_u^{(s(n))})}{f_{\mathbf{X}_u}(\mathbf{X}_u^{(s(n))})} - \widehat{p}_{t,N}^{\text{IS},\text{ub}},\end{equation} where $\left(s(n)\right)_{n\in[\![1,N_u]\!]}$ is a sample of uniformly distributed integers in $[\![1,N]\!]$ and where $k_N^{u}$ is defined in \Cref{sss:2.2.3}. This estimator does not require any additional call to the function $\phi$ to those from the already available i.i.d. sample and more generally, all the remarks previously made in \Cref{sss:MC_gd} about the double Monte Carlo given-data estimator by importance sampling \eqref{tevgd} of $\text{T-EV}_u$ are still valid. 

\begin{remark}\label{remark:3.1}
In all the above estimators \eqref{tevngd} \eqref{tevgd} \eqref{tvengd} and \eqref{tvegd}, for $v\in\lbrace u ,-u \rbrace$, it can be possible to draw a sample $\mathbf{X}_v^{(i)}$ according to the IS auxiliary distribution $g_{\mathbf{X}_v}$ such that $f_{\mathbf{X}_v}\left(\mathbf{X}_v^{(i)}\right) = 0$ which is at the denominator. However, this implies that for all $\mathbf{x}_{-v} \in \mathbb{X}_{-v}$, $f_{\mathbf{X}}\left(\mathbf{X}_v^{(i)},\mathbf{x}_{-v}\right) = 0$ and then $w_t^g\left(\mathbf{X}_v^{(i)},\mathbf{x}_{-v}\right) = 0$. Thus, the convention $0/0=0$ adopted at the beginning of the section allows to set as $0$ the term corresponding to this sample in all the previous estimators.
\end{remark}

\subsection{From reliability analysis to reliability-oriented sensitivity analysis}
\label{ss:3.3}

In practice, the ROSA of a complex system always comes after the reliability analysis, i.e. the estimation of the failure probability. When the latter has been done with importance sampling, we already have at our disposal a sub-optimal auxiliary PDF $g$ close to the optimal one $g_{\text{opt}}\left(\mathbf{x}\right)\propto\mathbf{1}\left(\phi\left(\mathbf{x}\right)>t\right)f_{\mathbf{X}}\left(\mathbf{x}\right)$ (see \Cref{ss:is}), as well as an i.i.d. $N$-sample distributed according to it. In order to check if it is beneficial to re-use the available data from the reliability analysis to estimate the target conditional indices and thus the target Shapley values, let us consider the variance of the previous new estimators when the auxiliary density is $g_{\text{opt}}$. Since it is based on a double loop, it may be complicated to write the variance of the double Monte-Carlo estimator \eqref{tevngd} in closed form, but this is feasible for the Pick-Freeze estimator \eqref{tvengd} since it is an empirical mean. Considering $N_u$ as fixed, we have (see \ref{app:proof_eq_var_pf_gopt}): \begin{equation}\label{eq_var_pf_gopt}
    \mathbb{V}_{g_{\text{opt}}}\left(\widehat{\text{T-VE}}_{u,\text{PF}}^{\text{IS}} \right) \leq \dfrac{p_t^2}{N_u} \propto p_t^2.
\end{equation} 

This result does not prove that $g_{\text{opt}}$ is the optimal IS auxiliary distribution to estimate the target closed Sobol index $\text{T-VE}_u$ from \eqref{tvengd}, but it proves that using $g_{\text{opt}}$ as the IS auxiliary density improves its estimation in comparison to its existing estimators without importance sampling in the regime where $p_t \to 0$. Indeed, letting $\widehat{\text{T-VE}}_{u,\text{PF}}$ be the unbiased estimator by Pick-Freeze without importance sampling of $\text{T-VE}_u$, defined as in \eqref{pf} with $\phi$ replaced by $\psi_t$, we have (see \ref{app:proof_eq_var_pf_fx}): \begin{equation}\label{eq_var_pf_fx}\mathbb{V}_{f_{\mathbf{X}}}\left(\widehat{\text{T-VE}}_{u,\text{PF}} \right) \geq \frac{1}{N_u}\left(\mathbb{V}_{f_{\mathbf{X}}}\left[\mathbb{E}_{f_{\mathbf{X}}}\left(\psi_t(\mathbf{X}) | \mathbf{X}_u\right)\right] + p_t^2\right)\left(1-\left(\mathbb{V}_{f_{\mathbf{X}}}\left[\mathbb{E}_{f_{\mathbf{X}}}\left(\psi_t(\mathbf{X}) | \mathbf{X}_u\right)\right] + p_t^2\right)\right).\end{equation} Thus, at best, when $\phi$ depends only on $\mathbf{X}_{-u}$, we have $\mathbb{V}_{f_{\mathbf{X}}}\left[\mathbb{E}_{f_{\mathbf{X}}}\left(\psi_t(\mathbf{X}) | \mathbf{X}_u\right)\right] = 0$ and then $\mathbb{V}_{f_{\mathbf{X}}}\left(\widehat{\text{T-VE}}_{u,\text{PF}} \right) \geq p_t^2\left(1-p_t^2\right)/N_u \propto p_t^2\left(1-p_t^2\right) \underset{p_t\to 0}{\sim} p_t^2$. At worst, when $\phi$ depends only on $\mathbf{X}_{u}$, we have $\mathbb{V}_{f_{\mathbf{X}}}\left[\mathbb{E}_{f_{\mathbf{X}}}\left(\psi_t(\mathbf{X}) | \mathbf{X}_u\right)\right] = p_t\left(1-p_t\right)$ and then $\mathbb{V}_{f_{\mathbf{X}}}\left(\widehat{\text{T-VE}}_{u,\text{PF}} \right) \geq p_t\left(1-p_t\right)/N_u \propto p_t\left(1-p_t\right)\underset{p_t\to 0}{\sim} p_t$. Since $p_t$ is the probability of a rare event, $p_t \ll 1$ and so $p_t^2\ll p_t$, thus using $g_{\text{opt}}$ as the IS auxiliary density improves the estimation of all the target closed Sobol indices in the regime where $p_t \longrightarrow 0$, and then of the $d$ target Shapley effects. 

However, in practice, the available sample is not drawn exactly from $g_{\text{opt}}$ but from an IS auxiliary distribution $g$ close to $g_{\text{opt}}$, but we can still expect a significant improvement in the estimation of the target closed Sobol indices. Hence, this theoretical analysis supports the intuition that it is beneficial to re-use the available data from the reliability analysis to estimate the target Shapley effects. 
\begin{remark}\label{remark:3.2}
The input domain $\mathbb{X}$ is not necessarily equal to $\mathbb{R}^d$. Nevertheless, it can be practically convenient to use an IS auxiliary distribution supported on $\mathbb{R}^d$ (or more generally on a subset of $\mathbb{R}^d$ strictly greater than $\mathbb{X}$) such as a normal distribution for example. To that end, the solution adopted in this article consists in extending the input domain of $\phi$ and $f_{\mathbf{X}}$ on $\mathbb{R}^d$ by setting $\phi\left(\mathbf{x}\right) = 0$ and $f_{\mathbf{X}}\left(\mathbf{x}\right) = 0$ for any $\mathbf{x}\in\mathbb{R}^d\backslash\mathbb{X}$. The convention $0/0=0$ allows then to ignore the samples drawn by $g$ on $\mathbb{R}^d\backslash\mathbb{X}$ during the estimation process (see \Cref{remark:3.1}). However, if the reliability analysis has been done efficiently, the IS auxiliary distribution $g$ should be close to $g_{\text{opt}}\left(\mathbf{x}\right) \propto \mathbf{1}\left(\phi\left(\mathbf{x}\right)>t\right)f_{\mathbf{X}}(\mathbf{x)}$ and therefore very few samples should be drawn in $\mathbb{R}^d\backslash\mathbb{X}$.
\end{remark}

\section{Numerical applications}
\label{sec:numerical_experiments}

In order to illustrate the practical interest of the previous efforts, this section aims to evaluate numerically the performances of the suggested estimators of the target Shapley effects on various test functions with correlated input variables and to compare them to the performances of the existing estimators. The code to reproduce the numerical experiments is publicly available at: \url{https://github.com/Julien6431/Target-Shapley-effects.git}

In the following examples, as explained in \Cref{ss:3.3}, we will consider that the reliability analysis has already been done, i.e. that an IS auxiliary density $g$ close to $g_{\text{opt}}$ has been determined. The present article does not aim to compare different importance sampling techniques, therefore we choose here to compute the IS auxiliary PDF $g$ only by adaptive importance sampling with the cross-entropy algorithm \cite{rubinstein2013cross}, from one of the two following IS parametric families: the Gaussian distributions (single Gaussian, IS-SG) \cite{rubinstein2013cross} and the Gaussian mixture (IS-GM) distributions~\cite{geyer2019cross}. 

Moreover, the dimension of the following problems will be low or moderate, therefore only the subset aggregation procedure will be used, and we adopt then the following numerical parameters:\begin{itemize}
    \item $N_{tot} = 2\times 10^4$ which represents the total number of calls to $\phi$
    \item $N_I = 3$ in \Cref{sec:linear_gaussian,ss:cantilever_beam} as suggested in \cite{song2016shapley,broto2020variance}, $N_I=2$ in \Cref{ss:fire_spread}, the size of the inner loop in the double Monte Carlo estimator
    \item $N_{\mathbb{V}} = 10^4$ the size of the sample to estimate $\mathbb{V}\left(\psi_t\left(\mathbf{X}\right)\right)$ in the given-model framework
    \item in the given-model framework, for $u \in \mathcal{P}(d)\backslash\lbrace\varnothing,[\![1,d]\!]\rbrace$, $N_u = N_O = \left\lfloor\left(N_{tot} - N_{\mathbb{V}}\right)\left(N_I(2^d-2)\right)^{-1}\right\rfloor$ with the double Monte Carlo estimators and $N_u = N_O = \left\lfloor\left(N_{tot} - N_{\mathbb{V}}\right)\left(2(2^d-2)\right)^{-1}\right\rfloor$ with the Pick-Freeze estimators, in order to reach $N_{tot}$ calls or less to $\phi$ (according to both expressions of $N_{tot}$ given in \Cref{ss:aggreg})
    \item in the given-data framework, for all $u \in \mathcal{P}(d)\backslash\lbrace\varnothing,[\![1,d]\!]\rbrace$, $N_u = N_O = 10^3$
    \item $n_{rep} = 200$ realisations of each estimator to represent the results as boxplots.
\end{itemize}
For different reasons described in more details in \ref{app:standardisation}, the preprocessing procedure presented in \ref{ss:preprocessing} will be applied as soon as a given-data estimator will be used. It is based on \Cref{invariance_bij} stated by~\cite{owen2017shapley}. 

In the following, results will be presented as boxplots. Let us define first the acronyms that will be used in the legends: \begin{itemize}
\item \textit{IS-SG} refers to estimators with importance sampling using an auxiliary distribution in the single Gaussian family determined by the cross-entropy algorithm,
\item \textit{IS-GM} refers to estimators with importance sampling using an auxiliary distribution in the Gaussian mixture family determined by the cross-entropy algorithm. \end{itemize} Then, each box extends from the first to the third quartile values of the data, with a line at the median. The whiskers extend from the box to show the range of the data, and flier points are those past the end of the whiskers.

\subsection{Gaussian linear case}
\label{sec:linear_gaussian}

First, let us consider the simple Gaussian linear example. For $d\geq2$ and a vector ${\boldsymbol{\beta}} \in \mathbb{R}^d\backslash\lbrace0\rbrace$, let us define the $d$-dimensional linear function $\phi_{\boldsymbol{\beta}}$ by: \begin{equation}
\begin{array}{l|rcl}
\phi_{\boldsymbol{\beta}}: & \mathbb{R}^d & \longrightarrow & \mathbb{R} \\
    & \mathbf{x} & \longmapsto & {\boldsymbol{\beta}}^\top \mathbf{x}.
    \end{array}
\end{equation} Then, the input vector $\mathbf{X}$ is assumed normally distributed with mean vector $\boldsymbol{\mu} \in \mathbb{R}^d$ and covariance matrix $\mathbf{\Sigma} \in \mathcal{M}_d\left(\mathbb{R}\right)$ which is symmetric positive-definite, i.e. $\mathbf{X} \sim \mathcal{N}_d\left(\boldsymbol{\mu},\mathbf{\Sigma}\right)$. Moreover the covariance matrix $\mathbf{\Sigma}$ is here considered non-diagonal in order to include dependence between the input variables. Finally, since the failure domain is here clearly located in one region of $\mathbb{R}^d$, only Gaussian auxiliary IS distributions will be considered.

The parameters of the toy case are specified by $\boldsymbol{\beta}  =  \begin{pmatrix}
1 & 1 & 1
\end{pmatrix}^\top$, $\boldsymbol{\mu}  =  \begin{pmatrix}
0 & 0 & 0
\end{pmatrix}^\top$, $\mathbf{\Sigma}  =  \begin{pmatrix}
1 & 0 & 0\\
0 & 1 & -0.3 \\
0 & -0.3 & 1
\end{pmatrix}$ and $t=4$. The failure threshold is set at $t=4$ such that the failure probability is $p_t^{\boldsymbol{\beta}}\approx 4.9\times 10^{-3}$. Reference values of the target Shapley effects will be computed with \eqref{soboltheo}, \eqref{proba_theo} and the definition in \eqref{shapley}. The performances of the estimators with and without importance sampling in both given-model and given-data frameworks are compared graphically respectively in \Cref{fig:gl_d3_ngd} and \Cref{fig:gl_d3_gd} (recall that the acronyms are defined above). As expected, in both frameworks, when the IS auxiliary density is adapted to the problem, the estimators with importance sampling give a much better estimation of the target Shapley effects when the failure probability is small. Despite few outliers, the corresponding boxplots have a much smaller stretch and they are centered on the reference values. This first example highlights then that importance sampling has a huge positive impact on the quality of the estimation of the target Shapley effects with a low failure probability. \begin{figure}[!tbp]
  \centering
  \begin{minipage}[b]{0.47\textwidth}
    \includegraphics[width=\textwidth]{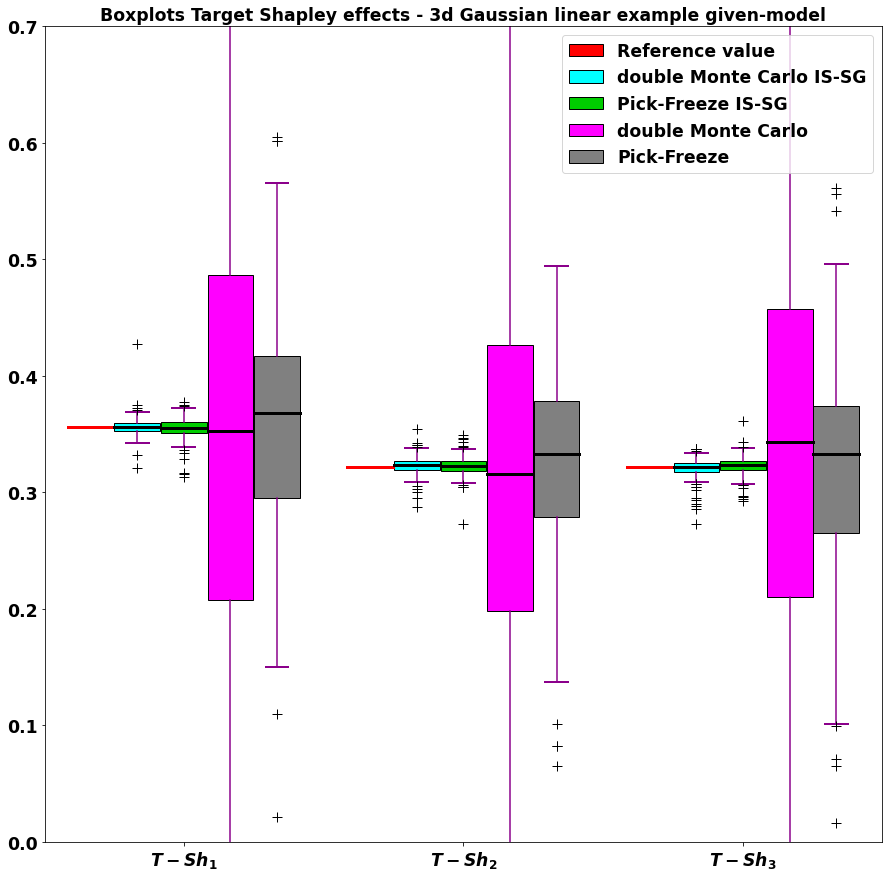}
    \caption{Estimation of the target Shapley effects in the $3$-dimensional Gaussian linear example in the given-model framework.}
    \label{fig:gl_d3_ngd}
  \end{minipage}
  \hfill
  \begin{minipage}[b]{0.47\textwidth}
    \includegraphics[width=\textwidth]{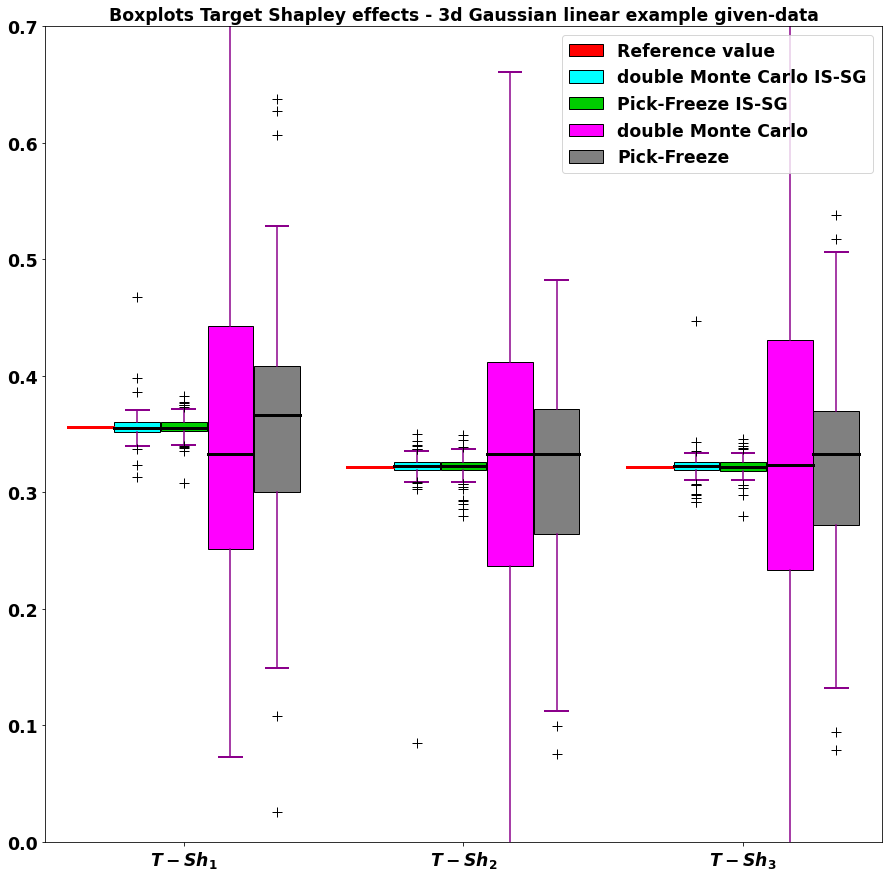}
    \caption{Estimation of the target Shapley effects in the $3$-dimensional Gaussian linear example in the given-data framework.}
    \label{fig:gl_d3_gd}
  \end{minipage}
\end{figure}

\subsection{Cantilever beam}
\label{ss:cantilever_beam}

\subsubsection{Presentation of the model}

The second example is a real structure engineering problem which is presented in \cite{zhou2014moment,baoyu2017reliability}. Consider a rectangular cantilever beam structure. The dimensional parameters of the beam are denoted $l_X$, $l_Y$ and $L$. The elastic modulus of the structure is represented by $E$. Two random forces $F_X$ and $F_Y$ are exerted on the tip of the section. The goal function is then the maximum vertical displacement of the tip section, which can be given analytically according to the previous parameters by: \begin{equation}
    \phi\left(F_X,F_Y,E,l_X,l_Y,L\right) = \dfrac{4L^3}{El_Xl_Y}\sqrt{\left(\dfrac{F_X}{l_X^2}\right)^2+\left(\dfrac{F_Y}{l_Y^2}\right)^2}.
\end{equation} The maximum vertical displacement allowed is $t=0.066 \ \mathrm{m}$, which is then the failure threshold of the reliability problem.

\begin{table}[h!]
\centering
\begin{tabular}{|c | l l l l|} 
 \hline
 & Input variable & Distribution & Mean & Coefficient of variation \\ [0.5ex] 
 \hline\hline
 1 & $F_X$ & LogNormal & $556.8 \ \mathrm{N}$ & $0.08$ \\ 
 2 & $F_Y$ & LogNormal & $453.6 \ \mathrm{N}$ & $0.08$ \\
 3 & $E$ & LogNormal & $200.10^9 \ \mathrm{Pa} $ & $0.06$ \\
 4 & $l_X$ & Normal & $0.062 \ \mathrm{m}$ & $0.1$ \\
 5 & $l_Y$ & Normal & $0.0987 \ \mathrm{m}$ & $0.1$ \\ 
 6 & $L$ & Normal & $4.29 \ \mathrm{m}$ & $0.1$ \\[1ex] 
 \hline
\end{tabular}
\caption{Distributions of each input variable of the cantilever beam example.}
\label{table:distr_cantiler_beam}
\end{table}The distributions of each input variable are listed in \Cref{table:distr_cantiler_beam}. Moreover, the dimensional variables are linearly dependent through the following Pearson correlation coefficients: \begin{equation}
    \rho_{l_X,l_Y} = -0.55 \mbox{ and } \rho_{L,l_X} = \rho_{L,l_Y} = 0.45.
\end{equation} 

We do not know anything about the form of the failure domain, therefore both Gaussian and Gaussian mixture IS auxiliary distributions will be used in order to evaluate the influence of the IS auxiliary density on the estimation of the target Shapley effects with importance sampling. First, we will compare the performances of the given-data estimators, according to the framework described in \Cref{ss:3.3}. Then, in order to evaluate graphically the error due to the nearest neighbour approximation, we will evaluate the performances of the given-model estimators with importance sampling with a Gaussian IS auxiliary distribution. Since the input vector $\mathbf{X} = \left(F_X,F_Y,E,l_X,l_Y,L\right)^\top$ is not Gaussian, we will not generate samples according to the input conditional distributions and thus we will not evaluate the performances of the existing given-model estimators without importance sampling.

\subsubsection{Procedure to obtain the reference values}

A Monte Carlo estimation with a sample of size $N = 10^6$ gives a reference value for the failure probability: $p_t^{\text{ref}} \approx 1.5\times10^{-2}$. Next, one can remark that a random variable following a LogNormal distribution can be written as a bijective transformation of a Gaussian random variable. Then, thanks to \Cref{invariance_bij}, we decide to transform the input random variable $\mathbf{X} = \left(F_X,F_Y,E,l_X,l_Y,L\right)$ into a $6$-dimensional Gaussian random vector with the correct mean vector and covariance matrix. At last, reference values of the target Shapley effects are computed using the double Monte Carlo existing estimator in \eqref{dmc} combined with the non-given data sampling procedure described in \ref{sss:3.3.2} with the input distribution and finally, with $N = 10^6$, $N_O = 10^6$ and $N_I=3$, we obtain the reference values presented in \Cref{tab:ref_values_TSA_cantilever_beam}. \begin{table}[]
    \centering
    \begin{tabular}{|c|c|c|c|c|c|}
   \hline
   $\text{T-Sh}_1^{\text{ref}}$ & $\text{T-Sh}_2^{\text{ref}}$ &$\text{T-Sh}_3^{\text{ref}}$ &$\text{T-Sh}_4^{\text{ref}}$ &$\text{T-Sh}_5^{\text{ref}}$ &$\text{T-Sh}_6^{\text{ref}}$ \\[0.5ex] 
   \hline\hline
   0.146&0.001&0.103&0.282&0.254&0.214\\
   \hline
\end{tabular}
    \caption{Reference values of the target Shapley effects in the cantilever beam problem.}
    \label{tab:ref_values_TSA_cantilever_beam}
\end{table}

\subsubsection{Numerical results}

The results of the ROSA of the cantilever beam problem in the given-data and given-model frameworks are respectively given in \Cref{fig:cantilever_beam_std,fig:cantilever_beam_ngd}. Once more, when the IS auxiliary density is adapted to the problem, the estimators with importance sampling give a better estimation of the target Shapley effects than the existing estimators for the same reasons as in the previous example. In addition, unreported numerical simulations provide that the failure domain is here located in one region of the input space which explains why both Gaussian and Gaussian mixture IS auxiliary densities provide similar performances. Both figures show that the dimensional parameters of the beam are the most influential inputs on the occurrence of the failure and show that the estimators with importance sampling provide the same hierarchy of importance of the inputs as the reference values of the target Shapley effects whereas the existing estimators without importance sampling switch the importance of $\text{T-Sh}_4$ and $\text{T-Sh}_5$.

In addition, one can remark that the dispersion of each given-model boxplot in \Cref{fig:cantilever_beam_ngd} is bigger than the dispersion of each given-data boxplot in \Cref{fig:cantilever_beam_std}. In fact, in both cases, the maximal number of calls to $\phi$ allowed is $N_{tot} = 2\times 10^4$. In the given-model framework, it is necessary to choose the parameters $N_I$, $N_O$ and $N_{\mathbb{V}}$ such that we exactly reach $N_{tot}$ calls to the function, whereas in the given-data framework, we are free to chose $N_I$ and $N_O$ as we want because we already have the $N_{tot}$-sample at our disposal. In both frameworks, the value of $N_I$ is already fixed. As a consequence, in the given-model framework, the value of $N_O$ must be equal to the one given in the introduction of Section 4, which is always lower than $10^2$ in each numerical example here, whereas in the given-data framework, we choose $N_O=10^3$. This gap between the value of $N_O$ in both frameworks explains the larger dispersion observed in the given-model algorithms.

Moreover, \Cref{fig:cantilever_beam_std} illustrates a problem already mentioned in \Cref{sss:MC_gd}. On some indices in \Cref{fig:cantilever_beam_std}, there is a gap between the boxplot median and the reference value whereas the boxplots of the given-model estimators with importance sampling presented in \Cref{fig:cantilever_beam_ngd} are centered on the reference values. Indeed, when the dimension increases, distances between points tend to become larger and thus the nearest neighbour approximations of the conditional distributions are getting less accurate. This phenomenon might create a bias in the estimation of the target Shapley effects with the given-data estimators when the dimension increases. However, one can remark on \Cref{fig:cantilever_beam_std} that importance sampling seems to reduce this error. Indeed, without importance sampling, the points of interest, the failure points, are in the tail of the distribution, where the concentration of points is small and thus where the distances between points are larger than on average, which is not the case with importance sampling when the auxiliary distribution is adapted to the problem. This is another advantage of using importance sampling to estimate the target Shapley effects. Finally, as seen in \Cref{fig:cantilever_beam} in \ref{app:standardisation}, the preprocessing introduced and described in \ref{ss:preprocessing} seems as well to reduce significantly the error.

\begin{figure}
    \centering
    \includegraphics[width=.6\textwidth]{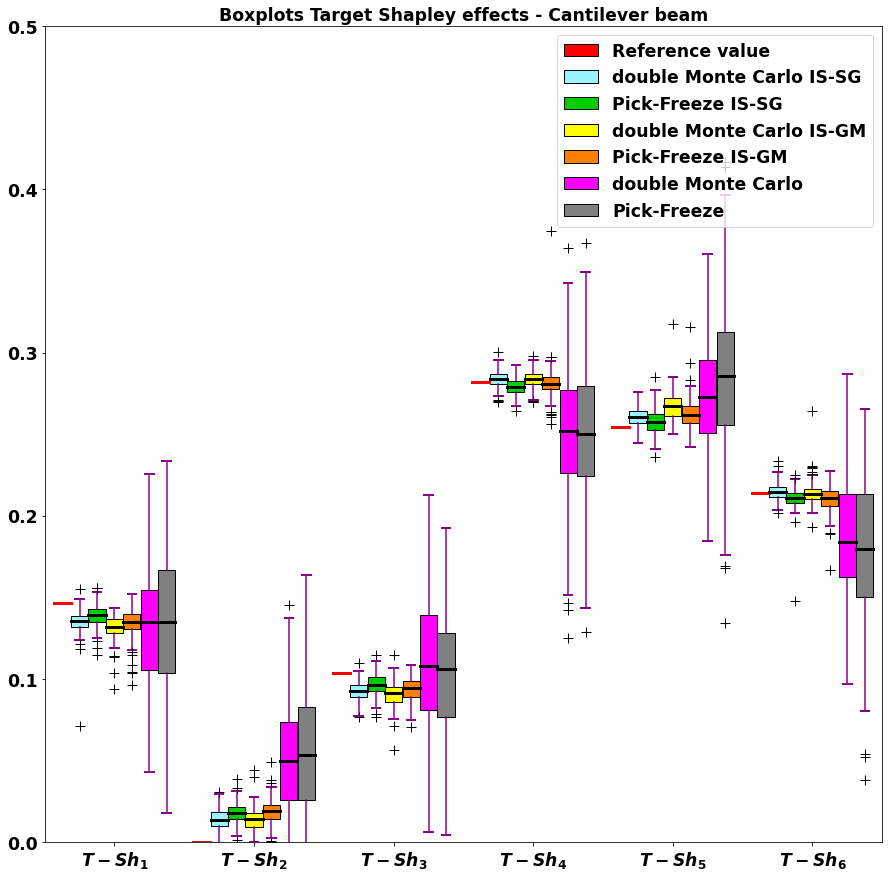}
    \caption{Estimation of the target Shapley effects in the cantilever beam example, in the given-data framework and with the preprocessing described in \ref{ss:preprocessing}.}
    \label{fig:cantilever_beam_std}
\end{figure}

\begin{figure}
    \centering
    \includegraphics[width=.6\textwidth]{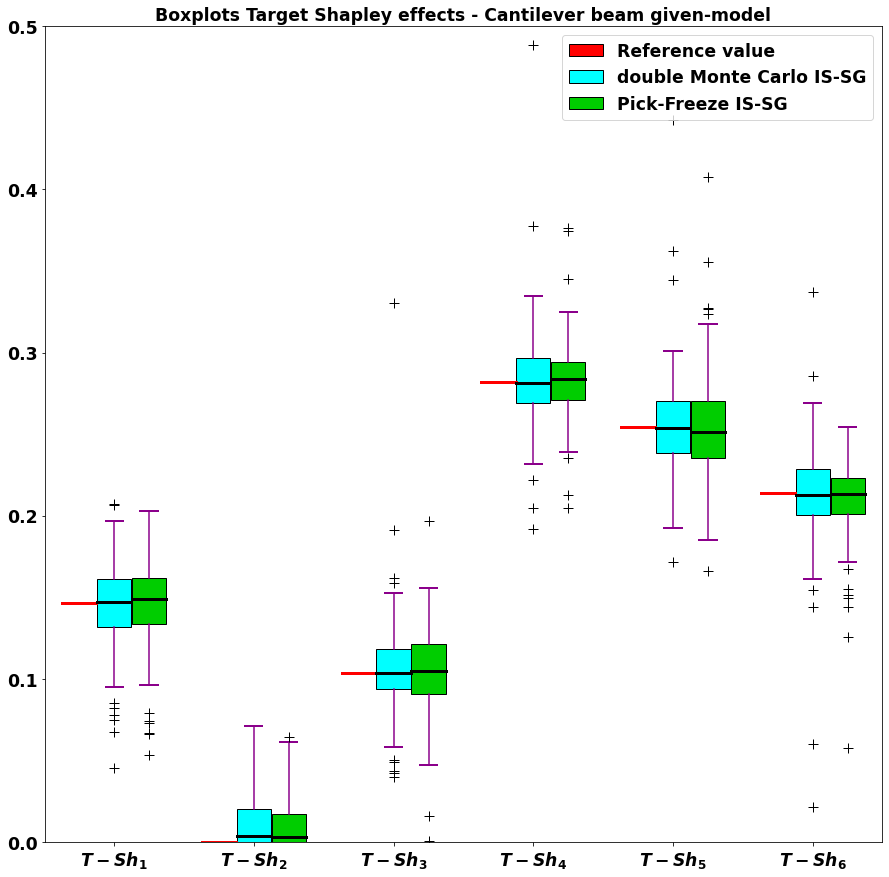}
    \caption{Estimation of the target Shapley effects in the cantilever beam example, in the given-model framework.}
    \label{fig:cantilever_beam_ngd}
\end{figure}

\subsection{Fire spread}
\label{ss:fire_spread}
\subsubsection{Presentation of the model}

The Rothermel's model introduced in \cite{rothermel1972mathematical} is a semi-physical model which aims at modeling the spread of forest fires. It has two main outputs: the rate of spread of a point in the fire front ($R$ given in $\mathrm{cm\cdot s}^{-1}$) and the reaction intensity ($I_R$ given in $\mathrm{kW\cdot m}^{-2}\cdot \mathrm{s}^{-1}$). The initial equations from \cite{rothermel1972mathematical} have been modified several times since their introduction and in the present study, we adopt the point of view of \cite{salvador2001global,song2016shapley}. In the same way, we take into account the modifications of \cite{albini1976estimating} on the net fuel loading and the optimum reaction velocity and the modifications of \cite{catchpole1991modelling} on the moisture damping coefficient and the heat preignition. In addition, we adopt as well the modifications of the marginal input distributions introduced and explained in \cite{song2016shapley}.

The model considered in the present article has $10$ input variables grouped in the random input vector $\mathbf{X} = \left(\delta,\sigma,h,\rho_p,m_l,m_d,S_T,U,\mathrm{tan} \ \varphi,P\right)$ and whose physical meanings as well as their marginal distributions are given in \Cref{table:distr_fire_spread}. \begin{table}[h!]
\centering
\begin{tabular}{|c|m{5.95cm} l l|} 
 \hline
 &Input variable & Symbol and unit & Distribution \\ [0.5ex] 
 \hline\hline
 1&Fuel depth & $\delta$ ($\mathrm{cm}$) & $\mathrm{LogN}\left(2.19,0.517\right)$  \\ 
 2&Fuel particle area-to-volume ratio & $\sigma$ ($\mathrm{cm}^{-1}$) & $\mathrm{LogN}\left(3.31,0.294\right)$ \\
 3&Fuel particle low heat content & $h$ ($\mathrm{Kcal\cdot kg}^{-1}$) & $\mathrm{LogN}\left(8.48,0.063\right)$  \\
 4&Oven-dry particle density & $\rho_p$ ($\mathrm{D\cdot W\cdot g \cdot cm}^{-3}$) & $\mathrm{LogN}\left(-0.592,0.219\right)$  \\
 5&Moisture content of the live fuel & $m_l$ ($\mathrm{H_2OgD\cdot W\cdot g}^{-1}$) & $\mathrm{N}\left(1.18,0.377\right)$  \\ 
 6&Moisture content of the dead fuel & $m_d$ ($\mathrm{H_2OgD\cdot W\cdot g}^{-1}$) & $\mathrm{N}\left(0.19,0.047\right)$ \\
 7&Fuel particle total mineral content & $S_T$ ($\mathrm{MIN\cdot gD\cdot W\cdot g}^{-1}$) & $\mathrm{N}\left(0.049,0.011\right)$ \\
 8&Wind speed at midflame height & $U$ ($\mathrm{km\cdot h}^{-1}$) & $6.9\mathrm{LogN}\left(1.0174,0.5569\right)$  \\
 9&Slope & $\mathrm{tan} \ \varphi$ & $\mathrm{N}\left(0.38,0.186\right)$  \\
 10&Dead fuel loading to total fuel loading & $P$ & $\mathrm{LogN}\left(-2.19,0.64\right)$  \\[1ex] 
\hline
\end{tabular}
\caption{Distributions of each input variable of the fire spread example. D.W.: dry weight - MIN: mineral weight - $\mathrm{N}\left(\mu,\sigma\right)$: the 1-dimensional normal distribution with mean $\mu\in\mathbb{R}$ and standard deviation $\sigma>0$ -  $a\mathrm{LogN}\left(\mu,\sigma\right)$: the distribution of $a\exp\left(A\right)$ with $A$ a 1-dimensional normal random variable of mean $\mu\in\mathbb{R}$ and standard deviation $\sigma>0$.}
\label{table:distr_fire_spread}
\end{table} Moreover, real observations in \cite{clark2008sensitivity} highlight a negative correlation between the moisture content of dead fuel $m_d$ and the wind speed $U$, i.e. the windier it is, the less moisture the dead fuels contain. We suppose here that the correlation is strong and that it is specified by the following Pearson correlation coefficient:\begin{equation}\label{rho_fire}
        \rho_{m_d,U} = -0.8.
\end{equation} The joint distribution of $\left(m_d,U\right)$ is then based on a Gaussian copula. Finally, in order to ensure the physical consistency of the model, the following rules are adopted:  \begin{itemize}
    \item all negative values of any input variable are rejected
    \item all values of $S_T$ and $P$ over $1$ are rejected
    \item all values of $m_d$ lower than $3/0.6$ are rejected since this value is the smallest possible surface area to volume ratio for fuels with a diameter less than $6$mm.
\end{itemize} This means that the input distribution is given by the distribution defined by \Cref{table:distr_fire_spread} and \eqref{rho_fire}, conditioned to the fact that none of the above rules lead to a rejection. In practice, we use truncated distributions in order to handle the inputs numerically. To sum up, given the random input vector $\mathbf{X} = \left(\delta,\sigma,h,\rho_p,m_l,m_d,S_T,\right.\allowbreak \left. U,\mathrm{tan} \ \varphi,P\right)$, the rate of spread is obtained through the system of equations in \ref{app:fire_spread_eqs}. Finally, the critical threshold of the rate of spread is arbitrarily set to $t=60 \ \mathrm{cm\cdot s}^{-1}$.

\subsubsection{Reference values}

A Monte Carlo estimation with a sample of size $N = 10^7$ gives a reference value for the failure probability: $p_t^{\text{ref}} \approx 1.4\times 10^{-4}$. Next, since the random input vector $\mathbf{X} = \left(\delta,\sigma,h,\rho_p,m_l,m_d,S_T,U,\mathrm{tan} \ \varphi,P\right)$ is composed of normal and LogNormal random variables, as in the previous example, we transform it into a $10$-dimensional Gaussian random vector with the correct mean vector and covariance matrix. At last, reference values of the target Shapley effects are computed using the double Monte Carlo existing estimator in \eqref{dmc} combined with the non-given data sampling procedure described in \ref{sss:3.3.2} with the input distribution and finally, with $N=10^7$, $N_O = 10^6$ and $N_I = 3$, we obtain the reference values presented in \Cref{tab:ref_values_TSA_fire_spread}. 

\begin{table}[]
    \centering
    \begin{tabular}{|c|c||c|c||c|c||c|c||c|c|}
   \hline
   $\text{T-Sh}_1^{\text{ref}}$ & 0.152 &$\text{T-Sh}_2^{\text{ref}}$ & 0.247 &$\text{T-Sh}_3^{\text{ref}}$ & 0.011 & $\text{T-Sh}_4^{\text{ref}}$ &0.003  &$\text{T-Sh}_{5}^{\text{ref}}$ & 0.162 \\[0.5ex] 
   \hline\hline
   $\text{T-Sh}_6^{\text{ref}}$ & 0.145&$\text{T-Sh}_7^{\text{ref}}$ & 0.016&$\text{T-Sh}_8^{\text{ref}}$ & 0.182& $\text{T-Sh}_9^{\text{ref}}$ & 0.009&$\text{T-Sh}_{10}^{\text{ref}}$ &0.073  \\
   \hline
\end{tabular}
    \caption{Reference values of the target Shapley effects in the fire spread problem.}
    \label{tab:ref_values_TSA_fire_spread}
\end{table}

\subsubsection{Numerical results}

The results of the ROSA of the fire spread problem in the given-data framework with $N_{tot} = 2\times 10^4$ are given in \Cref{fig:fire_spread}. First, since the failure probability is around $p_t^{\text{ref}} \sim 10^{-4}$, the existing estimators without importance sampling with samples of size $N_{\text{tot}} = 2\times 10^4$ drawn according to the input distribution return a value very close to $0$ almost every time because there are too few failure points. Thus, for the sake of conciseness, it is not worthy to show their boxplots. Second, unreported numerical simulations provide that the failure domain is once more located in one region of the input space and explain why both Gaussian and Gaussian mixture IS auxiliary densities provide similar performances.

\Cref{fig:fire_spread} shows the practical interest of the previous efforts on a real semi-physical model. Indeed, when the IS auxiliary density is adapted to the problem, the performances of the suggested estimators with importance sampling are satisfying because they have a low variance and a moderate bias. Note nevertheless that the hierarchy of importance of the inputs is not exactly the same for the reference values of the target Shapley effects. In contrast, the existing estimators can not give meaningful results as explained above. Comparing both results presented in \Cref{fig:fire_spread} and in \cite{song2016shapley}, one can remark that on the fire spread example, the five most influential inputs on the variability of the output $R$ and on the variability of the random variable $\mathbf{1}\left(R>60\right)$ are the same: the fuel depth $\delta$, the fuel particle area-to-volume ratio $\sigma$, the moisture contents of the live and dead fuel $m_l$ and $m_d$, and the wind speed at midflame height $U$.

Moreover, this test case highlights the importance of the choice of the size of the inner loop $N_I$ for the double Monte Carlo estimator in the given-data framework when the dimension is getting higher. Indeed, as in the previous examples, we first applied the double Monte Carlo estimator with the value $N_I=3$ as suggested in \cite{song2016shapley,broto2020variance}. However, the estimations were inaccurate: the variance of each estimator was extremely high and the estimation of each target Shapley effect was very often over $10^{40}$ in absolute value whereas the effect should theoretically lie between $0$ and $1$. Unreported numerical tests show that this phenomenon is getting even worse when $N_I$ increases. Recalling that $N_I$ is the number of nearest neighbours to find in the given-data framework, the origin of this problem seems to be once more the nearest neighbour approximation, which is getting less accurate when the dimension increases. With importance sampling, the error does not only come from the gap between the target value $\psi_t\left(\mathbf{X}^{(n_0)}_{u},\mathbf{X}^{(k_N^u(n_0,k))}_{-u}\right)$ and its approximation $\psi_t\left(\mathbf{X}^{(k_N^u(n_0,k))}\right)$ (for $n_0 \in [\![1,N]\!]$ and $k\in[\![1,N_I]\!]$) but also from the gap between the likelihood ratios evaluated in both points $f_{\mathbf{X}}\left(\mathbf{X}^{(k_N^u(n_0,k))}\right)\left/g\left(\mathbf{X}^{(k_N^u(n_0,k))}\right)\right.$ and $f_{\mathbf{X}}\left(\mathbf{X}^{(n_0)}_{u},\mathbf{X}^{(k_N^u(n_0,k))}_{-u}\right)\left/g\left(\mathbf{X}^{(n_0)}_{u},\mathbf{X}^{(k_N^u(n_0,k))}_{-u}\right)\right.$, which might become extremely large when the approximation is not accurate. To decrease this error, we hence decided to reduce the size of the inner loop to $N_I=2$ such that the double Monte Carlo algorithm has to find as many neighbours as in the Pick-Freeze algorithm, and the corresponding results presented in \Cref{fig:fire_spread} are much more satisfying.

\begin{figure}
    \centering
    \includegraphics[width=.6\textwidth]{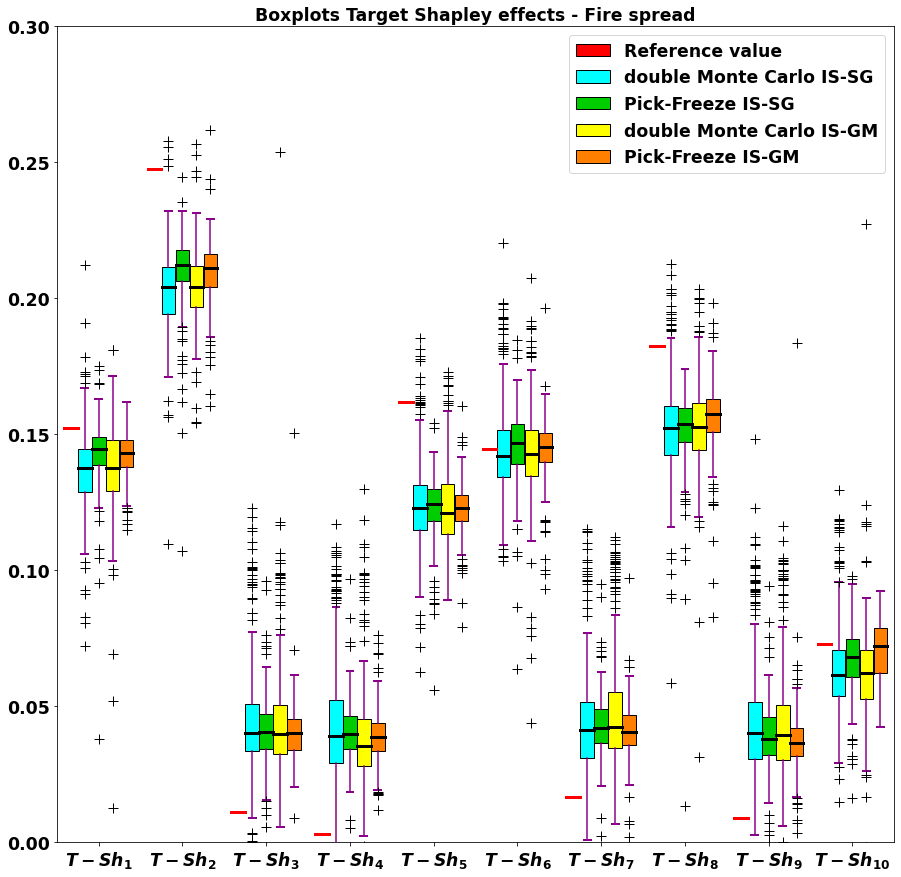}
    \caption{Estimation of the target Shapley effects in the fire spread example, in the given-data framework, with the preprocessing described in \ref{ss:preprocessing} and with $N_I = 2$.}
    \label{fig:fire_spread}
\end{figure}

In order to evaluate the error due to the nearest neighbour approximation, \Cref{fig:fire_spread_20000} represents the estimated ROSA indices for the fire spread problem in the given-model framework with $N_{tot} = 2\times 10^4$ and the numerical parameters from the beginning of \Cref{sec:numerical_experiments}. In contrast to the given-data results presented in \Cref{fig:fire_spread}, the reference value of each target Shapley effect is included in its corresponding boxplots and the boxplot medians are close to the reference value. The dispersion of the boxplots is nevertheless relatively large. \Cref{fig:fire_spread_100000} provides a deeper analysis and represents graphically the performances of the given-data and given-model estimators with importance sampling with samples of size $N_{tot} = 10^5$, with $N_{\mathbb{V}} = 10^4$ and $N_O$ as specified in the beginning of \Cref{sec:numerical_experiments} using the new value of $N_{tot}$ in the given-model framework. The given-model estimators provide the same hierarchy of importance of the inputs as the reference values and for each index, the boxplot medians are almost centered on the reference value, with moderate dispersion. In contrast, even if the bias of the given-data estimators seems to be smaller with larger samples, there is still a gap between the boxplot medians and the reference values, and the hierarchy of importance of the inputs provided is not exactly the same as the reference values. These observations reaffirm on an example in dimension $10$ the effect of the nearest neighbour approximation on the given-data estimators for the estimation of the target Shapley effects.

\begin{figure}[!tbp]
  \centering
  \begin{minipage}[b]{0.47\textwidth}
    \includegraphics[width=\textwidth]{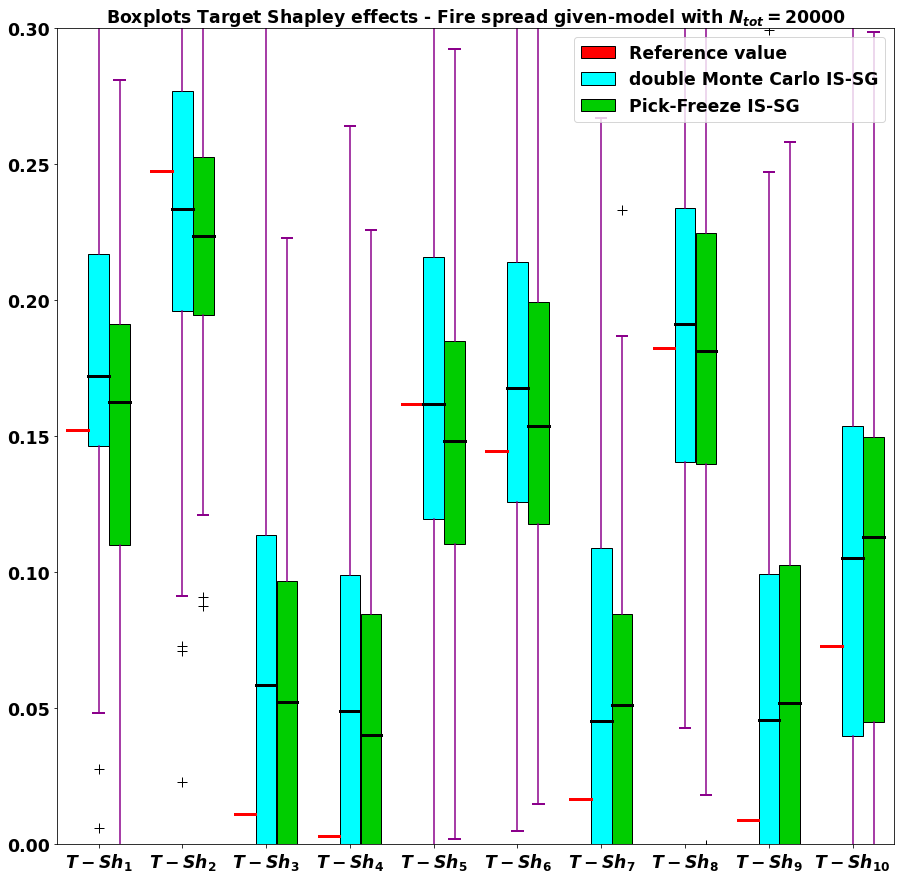}
    \caption{Estimation of the target Shapley effects in the fire spread example, in the given-model framework with $N_{tot} = 2\times 10^4$.}
    \label{fig:fire_spread_20000}
  \end{minipage}
  \hfill
  \begin{minipage}[b]{0.47\textwidth}
    \includegraphics[width=\textwidth]{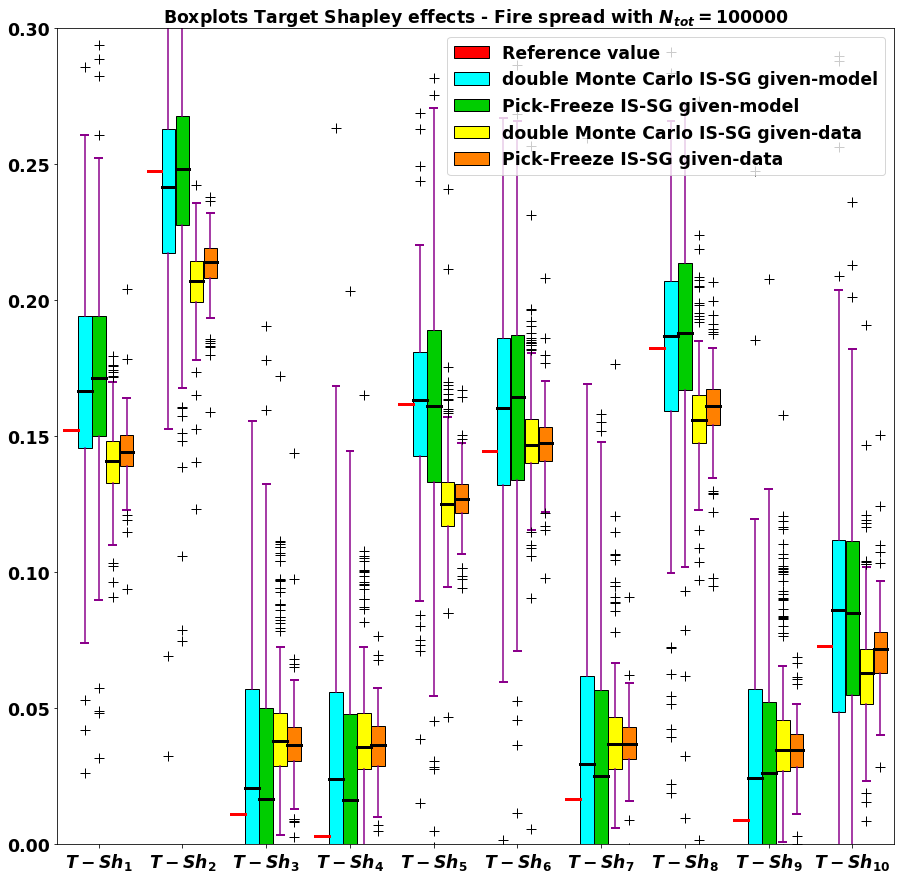}
    \caption{Estimation of the target Shapley effects in the fire spread example, in both given-model and given-data frameworks with $N_{tot} = 10^5$.}
    \label{fig:fire_spread_100000}
  \end{minipage}
\end{figure}

\begin{remark}
In all the figures, one can remark that there is no benefit to use the Pick-Freeze estimators instead of the double Monte Carlo estimators. This fact can be counter-intuitive knowing what happens in the independent case: the authors of \cite{janon2014asymptotic} proved that some Pick-Freeze estimators are asymptotically efficient. A similar remark has already been made in \cite{broto2020variance}, which can also be applied to our case. They highlighted first that the authors of \cite{janon2014asymptotic} estimate the variance of the output $Y$ in their procedure to estimate each Sobol indices, and second that the double Monte Carlo estimators are based on different observations from the Pick-Freeze estimators.\end{remark}

\FloatBarrier

\section{Conclusion}
\label{sec:conclusion}

In the present article, we are interested in the estimation of the target Shapley effects, whose goal is to quantify the influence of each input variable on the occurrence of the failure of the system and which are able to handle correlated inputs. We suggest new importance-sampling-based estimators of the target Shapley effects, extending the previous works of \cite{broto2020variance,ILIDRISSI2021105115}, which are more efficient than the existing ones when the failure probability is low. Moreover, we also introduce less expensive importance-sampling-based estimators requiring only an i.i.d. input/output $N$-sample distributed according to the IS auxiliary distribution, which enable to estimate efficiently the target Shapley effects without additional calls to the function $\phi$ after the estimation of the failure probability by importance sampling. In addition, we show theoretically that using the optimal IS auxiliary distribution for estimating a failure probability by importance sampling as the IS auxiliary distribution in the Pick-Freeze estimator improves the estimation of the target Shapley effects in comparison to the existing Pick-Freeze estimators. This result has a massive practical advantage because it justifies that it is beneficial to reuse the available sample from the reliability analysis to estimate the target Shapley effects by importance sampling. Finally, we illustrate and discuss the practical interest of the proposed estimators on the Gaussian linear case and on two real physical examples, all involving correlated inputs.

The main perspective for improvement of the suggested approach is to make the proposed estimators more robust faced with the dimension. Indeed, as explained and illustrated in the article, the nearest neighbour approximation creates an error which is getting larger when the dimension increases. The preprocessing procedure introduced in \ref{ss:preprocessing} seems to reduce this error, but potentially not enough so in higher dimension. Note also that in the highest dimension considered here, $d=10$ in \Cref{fig:fire_spread}, the performances provided by the nearest neighbour approximation are not as good as in the other settings. This classic problem in the analysis of a complex system is called the \textit{curse of dimensionality}. Nevertheless, new approaches based on projected random forests \cite{benard2022shaff} could improve the estimation of the target Shapley effects when the dimension increases. It might be possible to adapt the proposed method to our framework by taking into account the weights from importance sampling in the construction of the random forests. One can also mention recent projection methods \cite{zahm2018certified} to reduce the dimension of the problem.

Finally, it could be interesting to inspect methods to estimate the target Shapley effects efficiently while building a surrogate-model with importance sampling, in the same way as the method presented in \cite{echard2013combined}. At last, subset simulation could be used to estimate efficiently the target Shapley effects when the failure probability is very small instead of importance sampling. It may be done in low dimension by extending the work presented in \cite{perrin2019efficient} which aims to estimate the first and total orders target Sobol indices with a failure sample obtained by subset simulation.


\acknowledgements

The first author is enrolled in a Ph.D. program co-funded by \textit{ONERA – The French Aerospace Lab} and \textit{Toulouse III - Paul Sabatier University}. Their financial supports are gratefully acknowledged.






\appendix
\section{Proof of lemma 1}
\label{app:proof_lemma_fund_is_dmc}

First of all, let us remark that the convention $0/0=0$ introduced and adopted at the beginning of \Cref{sec:contribution} prevents the following proofs from any possible problem caused by a division by 0 or caused by a non-definition of any conditional PDF.

For any IS auxiliary density $g:\mathbb{X}\longrightarrow\mathbb{R}_+$, we have:
\begin{align*}
    \mathbb{E}_{f_{\mathbf{X}}}&\left[\mathbb{E}_{f_{\mathbf{X}}}\left(\psi_t(\mathbf{X})|\mathbf{X}_{-u}\right)^2 \right] = \int_{\mathbb{X}_{-u}} \mathbb{E}_{f_{\mathbf{X}}}\left(\psi_t(\mathbf{X})|\mathbf{X}_{-u} = \mathbf{x}_{-u}\right)^2f_{\mathbf{X}_{-u}}\left(\mathbf{x}_{-u}\right)d\mathbf{x}_{-u}\\
    &= \int_{\mathbb{X}_{-u}} \left(\int_{\mathbb{X}_u}\psi_t\left(\mathbf{x}_u,\mathbf{x}_{-u}\right) f_{\mathbf{X}_u|\mathbf{X}_{-u} = \mathbf{x}_{-u}}\left(\mathbf{x}_u\right)d\mathbf{x}_u\right)^2f_{\mathbf{X}_{-u}}\left(\mathbf{x}_{-u}\right)d\mathbf{x}_{-u} \\
    &= \int_{\mathbb{X}_{-u}} \left(\int_{\mathbb{X}_u}\psi_t\left(\mathbf{x}_u,\mathbf{x}_{-u}\right) \dfrac{f_{\mathbf{X}_u|\mathbf{X}_{-u}= \mathbf{x}_{-u}}\left(\mathbf{x}_u\right)}{g_{\mathbf{X}_u|\mathbf{X}_{-u}= \mathbf{x}_{-u}}\left(\mathbf{x}_u\right)}g_{\mathbf{X}_u|\mathbf{X}_{-u} = \mathbf{x}_{-u}}\left(\mathbf{x}_u\right)d\mathbf{x}_u\right)^2\dfrac{f_{\mathbf{X}_{-u}}\left(\mathbf{x}_{-u}\right)}{g_{\mathbf{X}_{-u}}\left(\mathbf{x}_{-u}\right)}g_{\mathbf{X}_{-u}}\left(\mathbf{x}_{-u}\right)d\mathbf{x}_{-u}\\
    &= \int_{\mathbb{X}_{-u}} \left[\mathbb{E}_g\left(\psi_t(\mathbf{X})\left.\dfrac{f_{\mathbf{X}_u|\mathbf{X}_{-u}}\left(\mathbf{X}_u\right)}{g_{\mathbf{X}_u|\mathbf{X}_{-u}}\left(\mathbf{X}_u\right)}\right|\mathbf{X}_{-u} = \mathbf{x}_{-u}\right)\right]^2\dfrac{f_{\mathbf{X}_{-u}}\left(\mathbf{x}_{-u}\right)}{g_{\mathbf{X}_{-u}}\left(\mathbf{x}_{-u}\right)}g_{\mathbf{X}_{-u}}\left(\mathbf{x}_{-u}\right)d\mathbf{x}_{-u}\\
    &= \mathbb{E}_g\left[\mathbb{E}_g\left(\psi_t(\mathbf{X})\left.\dfrac{f_{\mathbf{X}_u|\mathbf{X}_{-u}}\left(\mathbf{X}_u\right)}{g_{\mathbf{X}_u|\mathbf{X}_{-u}}\left(\mathbf{X}_u\right)}\right|\mathbf{X}_{-u}\right)^2 \dfrac{f_{\mathbf{X}_{-u}}\left(\mathbf{X}_{-u}\right)}{g_{\mathbf{X}_{-u}}\left(\mathbf{X}_{-u}\right)}\right].
\end{align*} Then, using the definition of the conditional PDF, remark that: $$\forall \left(\mathbf{x}_u,\mathbf{x}_{-u}\right) \in \mathbb{X}_u\times\mathbb{X}_{-u}, \ \left\{
    \begin{array}{rcl}
        f_{\mathbf{X}_u|\mathbf{X}_{-u}=\mathbf{x}_{-u}}\left(\mathbf{x}_u\right) &=& \dfrac{f_{\mathbf{X}}\left(\mathbf{x}_u,\mathbf{x}_{-u}\right)}{f_{\mathbf{X}_{-u}}\left(\mathbf{x}_{-u}\right)} \\
        g_{\mathbf{X}_u|\mathbf{X}_{-u}=\mathbf{x}_{-u}}\left(\mathbf{x}_u\right) &=& \dfrac{g\left(\mathbf{x}_u,\mathbf{x}_{-u}\right)}{g_{\mathbf{X}_{-u}}\left(\mathbf{x}_{-u}\right)}.
    \end{array}
\right. $$
By replacing in the above, we have: \begin{align*}
    \mathbb{E}_{f_{\mathbf{X}}}\left[\mathbb{E}_{f_{\mathbf{X}}}\left(\psi_t(\mathbf{X})|\mathbf{X}_{-u}\right)^2 \right] &= \mathbb{E}_g\left[\mathbb{E}_g\left(\psi_t(\mathbf{X})\left.\dfrac{f_{\mathbf{X}_u|\mathbf{X}_{-u}}\left(\mathbf{X}_u\right)}{g_{\mathbf{X}_u|\mathbf{X}_{-u}}\left(\mathbf{X}_u\right)}\right|\mathbf{X}_{-u}\right)^2 \dfrac{f_{\mathbf{X}_{-u}}\left(\mathbf{X}_{-u}\right)}{g_{\mathbf{X}_{-u}}\left(\mathbf{X}_{-u}\right)}\right] \\
    &= \mathbb{E}_g\left[\mathbb{E}_g\left(\psi_t(\mathbf{X})\left.\dfrac{f_{\mathbf{X}}\left(\mathbf{X}\right)}{f_{\mathbf{X}_{-u}}\left(\mathbf{X}_{-u}\right)}\dfrac{g_{\mathbf{X}_{-u}}\left(\mathbf{X}_{-u}\right)}{g\left(\mathbf{X}\right)}\right|\mathbf{X}_{-u}\right)^2 \dfrac{f_{\mathbf{X}_{-u}}\left(\mathbf{X}_{-u}\right)}{g_{\mathbf{X}_{-u}}\left(\mathbf{X}_{-u}\right)}\right]\\
    &= \mathbb{E}_g\left[\mathbb{E}_g\left(\psi_t(\mathbf{X})\left.\dfrac{f_{\mathbf{X}}\left(\mathbf{X}\right)}{g\left(\mathbf{X}\right)}\dfrac{g_{\mathbf{X}_{-u}}\left(\mathbf{X}_{-u}\right)}{f_{\mathbf{X}_{-u}}\left(\mathbf{X}_{-u}\right)}\right|\mathbf{X}_{-u}\right)^2 \dfrac{f_{\mathbf{X}_{-u}}\left(\mathbf{X}_{-u}\right)}{g_{\mathbf{X}_{-u}}\left(\mathbf{X}_{-u}\right)}\right]\\
    &= \mathbb{E}_g\left[\mathbb{E}_g\left(\psi_t(\mathbf{X})\left.\dfrac{f_{\mathbf{X}}\left(\mathbf{X}\right)}{g\left(\mathbf{X}\right)}\right|\mathbf{X}_{-u}\right)^2 \dfrac{g_{\mathbf{X}_{-u}}\left(\mathbf{X}_{-u}\right)^2}{f_{\mathbf{X}_{-u}}\left(\mathbf{X}_{-u}\right)^2} \dfrac{f_{\mathbf{X}_{-u}}\left(\mathbf{X}_{-u}\right)}{g_{\mathbf{X}_{-u}}\left(\mathbf{X}_{-u}\right)}\right]\\
    &=\mathbb{E}_g\left[\mathbb{E}_g\left(\psi_t(\mathbf{X})\left.\dfrac{f_{\mathbf{X}}\left(\mathbf{X}\right)}{g\left(\mathbf{X}\right)}\right|\mathbf{X}_{-u}\right)^2 \dfrac{g_{\mathbf{X}_{-u}}\left(\mathbf{X}_{-u}\right)}{f_{\mathbf{X}_{-u}}\left(\mathbf{X}_{-u}\right)}\right]\\
    &=\mathbb{E}_g\left[\mathbb{E}_g\left(w_t^g(\mathbf{X})\left.\right|\mathbf{X}_{-u}\right)^2 \dfrac{g_{\mathbf{X}_{-u}}\left(\mathbf{X}_{-u}\right)}{f_{\mathbf{X}_{-u}}\left(\mathbf{X}_{-u}\right)}\right].\end{align*} That concludes the proof of \Cref{fund_is_dmc}.\hfill $\square$

\section{Proof of proposition 1}
\label{app:proof_prop_dmc_is_ub}

First of all, recall that for $n\in[\![1,N_u]\!]$, the estimator $\overline{\psi_t^{\text{IS}}(\mathbf{X}_{-u}^{(n)})}$ in \eqref{inner_loop} is given by: \begin{equation}
   \overline{\psi_t^{\text{IS}}(\mathbf{X}_{-u}^{(n)})} = \dfrac{1}{N_I} \sum\limits_{k=1}^{N_I}w_t^g\left(\mathbf{X}_u^{(n,k)},\mathbf{X}_{-u}^{(n)}\right).
\end{equation}
Then, let us write $\widehat{E}_{u,\text{MC}}^{\text{IS}}$ for the uncorrected double Monte Carlo estimator of the term $\mathbb{E}_{f_{\mathbf{X}}}\left[\mathbb{E}_{f_{\mathbf{X}}}\left(\psi_t(\mathbf{X})|\mathbf{X}_{-u} \right)^2 \right] = \allowbreak \mathbb{E}_g\left[\mathbb{E}_g\left(w_t^g(\mathbf{X})\left.\right|\mathbf{X}_{-u}\right)^2 g_{\mathbf{X}_{-u}}\left(\mathbf{X}_{-u}\right)\left/f_{\mathbf{X}_{-u}}\left(\mathbf{X}_{-u}\right)\right.\right] $, obtained by removing $\widehat{E}_{\text{bias},u}^{\text{IS}}$ in \eqref{tevngd}:\begin{equation}\label{uncorr_dmc}
    \widehat{E}_{u,\text{MC}}^{\text{IS}} = \dfrac{1}{N_u}\sum_{n=1}^{N_u}\left(\overline{\psi_t^{\text{IS}}(\mathbf{X}_{-u}^{(n)})}\right)^2\dfrac{g_{\mathbf{X}_{-u}}\left(\mathbf{X}_{-u}^{(n)}\right)}{f_{\mathbf{X}_{-u}}\left(\mathbf{X}_{-u}^{(n)}\right)}.
\end{equation}  Then, this proof can be divided into two steps: \begin{enumerate}
    \item compute the bias of the estimator $\widehat{E}_{u,\text{MC}}^{\text{IS}}$ in \eqref{uncorr_dmc}
    \item propose an unbiased estimator of the latter bias in order to correct it.
\end{enumerate}  

\subsection{Bias in the inner loop}

First, for a given sample $\mathbf{X}_{-u}^{(n)} \in \mathbb{X}_{-u}$, let us compute the bias of the estimator $\left(\overline{\psi_t^{\text{IS}}(\mathbf{X}_{-u}^{(n)})}\right)^2$ in \eqref{inner_loop}: \begin{align*}
&\mathbb{E}_g\left(\left.\left(\overline{\psi_t^{\text{IS}}(\mathbf{X}_{-u}^{(n)})}\right)^2 \right|\mathbf{X}_{-u} =  \mathbf{X}_{-u}^{(n)}\right) = \mathbb{E}_g\left(\left.\left(\dfrac{1}{N_I} \sum\limits_{k=1}^{N_I}w_t^g\left(\mathbf{X}_u^{(n,k)},\mathbf{X}_{-u}^{(n)}\right)\right)^2 \right|\mathbf{X}_{-u} =  \mathbf{X}_{-u}^{(n)}\right) \\
&=\dfrac{1}{N_I^2} \sum_{k=1}^{N_I} \mathbb{E}_g\left(\left.w_t^g\left(\mathbf{X}_u^{(n,k)},\mathbf{X}_{-u}^{(n)}\right) ^2\right|\mathbf{X}_{-u} = \mathbf{X}_{-u}^{(n)}\right) \\
& \hspace{.25cm} + \dfrac{1}{N_I^2} \sum_{1\leq i\neq j \leq N_I} \mathbb{E}_g\left(\left. w_t^g\left(\mathbf{X}_u^{(n,i)},\mathbf{X}_{-u}^{(n)}\right)\right|\mathbf{X}_{-u} = \mathbf{X}_{-u}^{(n)}\right) \times \mathbb{E}_g\left(\left. w_t^g\left(\mathbf{X}_u^{(n,j)},\mathbf{X}_{-u}^{(n)}\right)\right|\mathbf{X}_{-u} = \mathbf{X}_{-u}^{(n)}\right)\\
&= \dfrac{1}{N_I} \mathbb{E}_g\left[\left.w_t^g\left(\mathbf{X}_u,\mathbf{X}_{-u}^{(n)}\right) ^2\right|\mathbf{X}_{-u} = \mathbf{X}_{-u}^{(n)} \right] + \dfrac{N_I-1}{N_I} \mathbb{E}_g\left[\left.w_t^g\left(\mathbf{X}_u,\mathbf{X}_{-u}^{(n)}\right)\right|\mathbf{X}_{-u} = \mathbf{X}_{-u}^{(n)} \right]^2\\
&= \mathbb{E}_g\left[\left.w_t^g\left(\mathbf{X}_u,\mathbf{X}_{-u}^{(n)}\right)\right|\mathbf{X}_{-u} = \mathbf{X}_{-u}^{(n)} \right]^2 \\
& \quad + \dfrac{1}{N_I}\mathbb{E}_g\left[\left.w_t^g\left(\mathbf{X}_u,\mathbf{X}_{-u}^{(n)}\right)^2\right|\mathbf{X}_{-u} = \mathbf{X}_{-u}^{(n)} \right] - \dfrac{1}{N_I} \mathbb{E}_g\left[\left.w_t^g\left(\mathbf{X}_u,\mathbf{X}_{-u}^{(n)}\right)\right|\mathbf{X}_{-u} = \mathbf{X}_{-u}^{(n)} \right]^2 \\
&= \mathbb{E}_g\left(\left.w_t^g\left(\mathbf{X}\right)\right|\mathbf{X}_{-u} = \mathbf{X}_{-u}^{(n)} \right)^2 + \dfrac{1}{N_I}\mathbb{V}_g\left(\left.w_t^g\left(\mathbf{X}\right)\right|\mathbf{X}_{-u} = \mathbf{X}_{-u}^{(n)}\right) \\
&= \mathbb{E}_g\left(\left.w_t^g\left(\mathbf{X}\right)\right|\mathbf{X}_{-u} = \mathbf{X}_{-u}^{(n)} \right)^2 + \mathbb{V}_g\left( \overline{\psi_t^{\text{IS}}(\mathbf{X}_{-u}^{(n)})}|\mathbf{X}_{-u} = \mathbf{X}_{-u}^{(n)}\right).
\end{align*} The bias of the estimator $\left(\overline{\psi_t^{\text{IS}}(\mathbf{X}_{-u}^{(n)})}\right)^2$ is thus $\mathbb{V}_g\left( \overline{\psi_t^{\text{IS}}(\mathbf{X}_{-u}^{(n)})}|\mathbf{X}_{-u} = \mathbf{X}_{-u}^{(n)}\right)$.

\subsection{Bias of the outer loop}

Second, let us derive the bias of the uncorrected double Monte Carlo estimator $\widehat{E}_{u,\text{MC}}^{\text{IS}}$ in \eqref{uncorr_dmc} of the double expectation $\mathbb{E}_g\left[\mathbb{E}_g\left(w_t^g(\mathbf{X})\left.\right|\mathbf{X}_{-u}\right)^2 g_{\mathbf{X}_{-u}}\left(\mathbf{X}_{-u}\right)\left/f_{\mathbf{X}_{-u}}\left(\mathbf{X}_{-u}\right)\right.\right]$ composed of both inner and outer loops: \begin{align*}
\mathbb{E}_g\left(\widehat{E}_{u,\text{MC}}^{\text{IS}} \right) &= \dfrac{1}{N_u} \sum_{n=1}^{N_u}\mathbb{E}_g\left(\left(\overline{\psi_t^{\text{IS}}(\mathbf{X}_{-u}^{(n)})}\right)^2\dfrac{g_{\mathbf{X}_{-u}}(\mathbf{X}_{-u}^{(n)})}{f_{\mathbf{X}_{-u}}(\mathbf{X}_{-u}^{(n)})}\right) \\
&= \dfrac{1}{N_u} \sum_{n=1}^{N_u}\mathbb{E}_g\left[\underbrace{\dfrac{g_{\mathbf{X}_{-u}}(\mathbf{X}_{-u}^{(n)})}{f_{\mathbf{X}_{-u}}(\mathbf{X}_{-u}^{(n)})} \mathbb{E}_g\left(\left.\left(\overline{\psi_t^{\text{IS}}(\mathbf{X}_{-u}^{(n)})}\right)^2\right|\mathbf{X}_{-u} = \mathbf{X}_{-u}^{(n)} \right)}_{\text{function of } \mathbf{X}_{-u}^{(n)} } \right] \\
&= \mathbb{E}_g\left[\dfrac{g_{\mathbf{X}_{-u}}(\mathbf{X}_{-u})}{f_{\mathbf{X}_{-u}}(\mathbf{X}_{-u})} \mathbb{E}_g\left(\left.\left(\overline{\psi_t^{\text{IS}}(\mathbf{X}_{-u})}\right)^2\right|\mathbf{X}_{-u} \right) \right ] \hspace{1cm} \text{where } \overline{\psi_t^{\text{IS}}} \text{ is as in \eqref{psi_t_is}} \\
&= \mathbb{E}_g\left[\dfrac{g_{\mathbf{X}_{-u}}(\mathbf{X}_{-u})}{f_{\mathbf{X}_{-u}}(\mathbf{X}_{-u})} \left(\mathbb{E}_g\left(\left.w_t^g\left(\mathbf{X}\right)\right|\mathbf{X}_{-u} \right)^2 + \mathbb{V}_g\left( \overline{\psi_t^{\text{IS}}(\mathbf{X}_{-u})}|\mathbf{X}_{-u}\right) \right)  \right] \\
&= \mathbb{E}_g\left[\dfrac{g_{\mathbf{X}_{-u}}(\mathbf{X}_{-u})}{f_{\mathbf{X}_{-u}}(\mathbf{X}_{-u})}\mathbb{E}_g\left(\left.w_t^g\left(\mathbf{X}\right)\right|\mathbf{X}_{-u} \right)^2 \right] + \mathbb{E}_g\left[\mathbb{V}_g\left( \overline{\psi_t^{\text{IS}}(\mathbf{X}_{-u})}|\mathbf{X}_{-u}\right)\dfrac{g_{\mathbf{X}_{-u}}(\mathbf{X}_{-u})}{f_{\mathbf{X}_{-u}}(\mathbf{X}_{-u})} \right]\\
&= \mathbb{E}_{f_{\mathbf{X}}}\left[\mathbb{E}_{f_{\mathbf{X}}}\left(\psi_t(\mathbf{X})|\mathbf{X}_{-u} \right)^2 \right] + \mathbb{E}_g\left[\mathbb{V}_g\left( \overline{\psi_t^{\text{IS}}(\mathbf{X}_{-u})}|\mathbf{X}_{-u}\right)\dfrac{g_{\mathbf{X}_{-u}}(\mathbf{X}_{-u})}{f_{\mathbf{X}_{-u}}(\mathbf{X}_{-u})} \right] \\
& \hspace{8cm} \text{ thanks to \Cref{fund_is_dmc}.}
\end{align*} Therefore, the bias of $\widehat{E}_{u,\text{MC}}^{\text{IS}}$ is $\mathbb{E}_g\left[\mathbb{V}_g\left( \overline{\psi_t^{\text{IS}}(\mathbf{X}_{-u})}|\mathbf{X}_{-u}\right)g_{\mathbf{X}_{-u}}(\mathbf{X}_{-u})\left/f_{\mathbf{X}_{-u}}(\mathbf{X}_{-u})\right. \right]$. The problem is now to estimate it in order to propose an unbiased estimator of $\text{T-EV}_u$ by double Monte Carlo with importance sampling.

\subsection{Estimation of the bias}

To estimate the previous bias, let us before prove the following lemma.\begin{lemma}\label{esti_var_est}
Let $(Z_n)_{n\in [\![1,N]\!]}$ be a sequence of independent and identically distributed random variables such that $\mathbb{E}\left(Z_1^2\right) < + \infty$. Let us consider the empirical estimator $\widehat{Z}_N = N^{-1}\sum_{n=1}^N Z_n$ of the mean value of $Z_1$. Then: \begin{equation}
    \widehat{V}_Z = \frac{1}{N-1}\left[\frac{1}{N}\sum_{n=1}^N Z_n^2 - \widehat{Z}_N^2\right]
\end{equation} is an unbiased estimator of $\mathbb{V}\left(\widehat{Z}_N\right)$.
\end{lemma}
 \begin{proof}Let us compute the expectation of the estimator $\widehat{V}_Z$:\begin{align*}
    \mathbb{E}\left(\widehat{V}_Z\right) &= \mathbb{E}\left(\frac{1}{N-1}\left[\frac{1}{N}\sum_{n=1}^N Z_n^2 - \widehat{Z}_N^2\right]\right)\\
    &= \frac{1}{N-1}\left[\mathbb{E}\left(Z^2\right) - \mathbb{E}\left(\widehat{Z}_N^2\right)\right]\\
    &= \frac{1}{N-1}\left[\mathbb{E}\left(Z^2\right) - \mathbb{E}\left(\frac{1}{N^2}\sum_{n=1}^N Z_n^2 + \frac{1}{N^2}\sum_{1\leq i \neq j \leq N} Z_iZ_j\right)\right]\\
    &= \frac{1}{N-1}\left[\mathbb{E}\left(Z^2\right) - \frac{1}{N}\mathbb{E}\left(Z^2\right) -  \frac{N-1}{N}\mathbb{E}\left(Z\right)^2\right]\\
    &= \frac{1}{N-1}\left[\frac{N-1}{N}\left(\mathbb{E}\left(Z^2\right)-\mathbb{E}\left(Z\right)^2\right)\right]\\
    &= \frac{N-1}{N(N-1)}\mathbb{V}(Z)\\
    &= \frac{1}{N}\mathbb{V}(Z)\\
    &= \mathbb{V}(\widehat{Z}_N).
\end{align*}That concludes the proof of this lemma.\end{proof}

By applying the previous lemma to the i.i.d. sequence $\left(w_t^g\left(\mathbf{X}_u^{(n,k)},\mathbf{X}_{-u}^{(n)}\right)\right)_{k\in[\![1,N_I]\!]}$ for all $n \in [\![1,N_u]\!]$, the estimator $\left(N_I-1\right)^{-1}\left[N_I^{-1} \sum\limits_{k=1}^{N_I}w_t^g\left( \mathbf{X}_u^{(n,k)},\mathbf{X}_{-u}^{(n)}\right)^2 - \left(\overline{\psi_t^{\text{IS}}\left(\mathbf{X}_{-u}^{(n)}\right)}\right)^2\right]$ estimates without bias the conditional variance $\mathbb{V}_g\left( \overline{\psi_t^{\text{IS}}(\mathbf{X}_{-u}^{(n)})}|\mathbf{X}_{-u} = \mathbf{X}_{-u}^{(n)}\right)$. Then, since the outer loop is only an empirical mean, $\widehat{E}_{\text{bias},u}^{\text{IS}}$ in \eqref{bias_est} is therefore an unbiased estimator of the bias $\mathbb{E}_g\left[\mathbb{V}_g\left( \overline{\psi_t^{\text{IS}}(\mathbf{X}_{-u})}|\mathbf{X}_{-u}\right)g_{\mathbf{X}_{-u}}(\mathbf{X}_{-u})\left/f_{\mathbf{X}_{-u}}(\mathbf{X}_{-u})\right. \right]$ and allows to correct the bias created by the square in the inner loop. Finally, by the linearity of the expectation, $\widehat{E}_{u,\text{MC}}^{\text{IS}} - \widehat{E}_{\text{bias},u}^{\text{IS}}$ is an unbiased estimator of $\mathbb{E}_g\left[\mathbb{E}_g\left(w_t^g(\mathbf{X})\left.\right|\mathbf{X}_{-u}\right)^2 g_{\mathbf{X}_{-u}}\left(\mathbf{X}_{-u}\right)\left/f_{\mathbf{X}_{-u}}\left(\mathbf{X}_{-u}\right)\right.\right]$, and thus $\widehat{\text{T-EV}}_{u,\text{MC}}^{\text{IS}}$ in \eqref{tevngd} is an unbiased estimator by double Monte Carlo with importance sampling of $\text{T-EV}_u$. That concludes the proof of \Cref{dmc_is_ub}.\hfill $\square$

\section{Proof that the estimator in eq. (25) is an unbiased estimator of $p_t^2$}
\label{app:ub_pt2}

Let us compute the bias of the estimator $\left(\widehat{p}_{t,N}^{\text{IS}}\right)^2$ in \eqref{pt2_is_ub}: \begin{align*}
    \mathbb{E}_g\left(\left(\widehat{p}_{t,N}^{\text{IS}}\right)^2\right) &=  \mathbb{E}_g\left(\left(\frac{1}{N}\sum_{n=1}^N w_t^g\left(\mathbf{X}^{(n)}\right)\right)^2\right)\\
    &= \dfrac{1}{N^2}\mathbb{E}_g\left(\sum_{n=1}^N w_t^g\left(\mathbf{X}^{(n)}\right)^2\right) + \dfrac{1}{N^2}\mathbb{E}_g\left(\sum_{1\leq i\neq j \leq N} w_t^g\left(\mathbf{X}^{(i)}\right)w_t^g\left(\mathbf{X}^{(j)}\right)\right)\\
    &= \dfrac{N}{N^2}\mathbb{E}_g\left(w_t^g\left(\mathbf{X}\right)^2\right) + \dfrac{N(N-1)}{N^2}\mathbb{E}_g\left(w_t^g\left(\mathbf{X}\right)\right)\mathbb{E}_g\left(w_t^g\left(\mathbf{X}\right)\right)\\
    &= \dfrac{1}{N}\mathbb{E}_g\left(w_t^g\left(\mathbf{X}\right)^2\right) + \dfrac{N-1}{N}\mathbb{E}_g\left(w_t^g\left(\mathbf{X}\right)\right)^2\\
    &= \mathbb{E}_g\left(w_t^g\left(\mathbf{X}\right)\right)^2 + \dfrac{1}{N}\left(\mathbb{E}_g\left(w_t^g\left(\mathbf{X}\right)^2\right) - \mathbb{E}_g\left(w_t^g\left(\mathbf{X}\right)\right)^2\right)\\
    &= \mathbb{E}_g\left(w_t^g\left(\mathbf{X}\right)\right)^2 + \dfrac{1}{N}\mathbb{V}_g\left(w_t^g\left(\mathbf{X}\right)\right)\\
    &= p_t^2 + \mathbb{V}_g\left(\widehat{p}_{t,N}^{\text{IS}}\right).
\end{align*} Then, \Cref{esti_var_est} justifies that $\left(N-1\right)^{-1}\left[N^{-1} \sum\limits_{n=1}^{N} w_t^g\left(\mathbf{X}^{(n)}\right)^2 - \left(\widehat{p}_{t,N}^{\text{IS}}\right)^2\right]$ is an unbiased estimator of $\mathbb{V}_g\left(\widehat{p}_{t,N}^{\text{IS}}\right)$. Therefore, by the linearity of the expectation, $\widehat{p}_{t,N}^{\text{IS},\text{ub}}$ in \eqref{pt2_is_ub} is an unbiased estimator of $p_t^2$. \hfill $\square$

\section{Proof of lemma 2}
\label{app:proof_lemma_fund_is_pf}
To begin with, let us prove the following lemma.
\begin{lemma}The PDF of the joint distribution of the random vector $\left(\mathbf{X}_u,\mathbf{X}_{-u},\mathbf{X}_{-u}'\right)$ satisfies for all $\mathbf{x}_u,\mathbf{x}_{-u},\mathbf{x}_{-u}' \in \mathbb{X}_{u}\times \mathbb{X}_{-u}\times \mathbb{X}_{-u}$: \begin{equation}f_{\mathbf{X}_u,\mathbf{X}_{-u},\mathbf{X}_{-u}'}(\mathbf{x}_u,\mathbf{x}_{-u},\mathbf{x}_{-u}') = f_{\mathbf{X}}(\mathbf{x}_u,\mathbf{x}_{-u}) \dfrac{f_{\mathbf{X}}(\mathbf{x}_u,\mathbf{x}_{-u}')}{f_{\mathbf{X}_u}(\mathbf{x}_u)}.\end{equation}\end{lemma}

\begin{proof}

Let $\mathbf{x}_u,\mathbf{x}_{-u},\mathbf{x}_{-u}' \in \mathbb{X}_{u}\times \mathbb{X}_{-u}\times \mathbb{X}_{-u}$. Then, regardless whether $\mathbf{x}_u$ is in the support of $f_{\mathbf{X}_u}$ or not, the convention $0/0=0$ allows to write:
\begin{align*}
f_{\mathbf{X}_u,\mathbf{X}_{-u},\mathbf{X}_{-u}'}(\mathbf{x}_u,\mathbf{x}_{-u},\mathbf{x}_{-u}') &= f_{\mathbf{X}_u}(\mathbf{x}_u) f_{\mathbf{X}_{-u},\mathbf{X}_{-u}'|\mathbf{X_u}=\mathbf{x}_u}(\mathbf{x}_{-u},\mathbf{x}_{-u}')\\
&= f_{\mathbf{X}_u}(\mathbf{x}_u) f_{\mathbf{X}_{-u}|\mathbf{X}_u=\mathbf{x}_u}(\mathbf{x}_{-u})f_{\mathbf{X}_{-u}'|\mathbf{X}_u=\mathbf{x}_u}(\mathbf{x}_{-u}')\\
&= f_{\mathbf{X}_u}(\mathbf{x}_u) \dfrac{f_{\mathbf{X}_u, \mathbf{X}_{-u}}(\mathbf{x}_u, \mathbf{x}_{-u})}{f_{\mathbf{X}_u}(\mathbf{x}_u)} \dfrac{f_{\mathbf{X}_u, \mathbf{X}_{-u}'}(\mathbf{x}_u, \mathbf{x}_{-u}')}{f_{\mathbf{X}_u}(\mathbf{x}_u)}\\
&=f_{\mathbf{X}}(\mathbf{x}_u,\mathbf{x}_{-u}) \dfrac{f_{\mathbf{X}}(\mathbf{x}_u,\mathbf{x}_{-u}')}{f_{\mathbf{X}_u}(\mathbf{x}_u)}.
\end{align*} That concludes the proof of the lemma.
\end{proof} 
Then, by remarking that the term in the expectation $\mathbb{E}_f\left(\psi_t(\mathbf{X})\psi_t(\mathbf{X}^u) \right)$ is a function of the three random variables $\mathbf{X}_u$, $\mathbf{X}_{-u}$ and $\mathbf{X}_{-u}'$ with the correlation structure described above, we have:
\begin{align*}
&\mathbb{E}_{f_{\mathbf{X}}}\left(\psi_t(\mathbf{X})\psi_t(\mathbf{X}^u) \right) = \mathbb{E}_{f_{\mathbf{X}}}\left(\psi_t\left(\mathbf{X}_u,\mathbf{X}_{-u} \right)\psi_t\left(\mathbf{X}_u,\mathbf{X}_{-u}' \right) \right) \\
&= \int_{\mathbb{X}_u}\int_{\mathbb{X}_{-u}}\int_{\mathbb{X}_{-u}} \psi_t(\mathbf{x}_u,\mathbf{x}_{-u})\psi_t(\mathbf{x}_u,\mathbf{x}_{-u}') f_{\mathbf{X}_u,\mathbf{X}_{-u},\mathbf{X}_{-u}'}(\mathbf{x}_u,\mathbf{x}_{-u},\mathbf{x}_{-u}') d\mathbf{x}_{-u}'d\mathbf{x}_{-u}d\mathbf{x}_{u}\\
&= \int_{\mathbb{X}_u}\int_{\mathbb{X}_{-u}}\int_{\mathbb{X}_{-u}} \psi_t(\mathbf{x}_u,\mathbf{x}_{-u})\psi_t(\mathbf{x}_u,\mathbf{x}_{-u}') \dfrac{f_{\mathbf{X}_u,\mathbf{X}_{-u},\mathbf{X}_{-u}'}(\mathbf{x}_u,\mathbf{x}_{-u},\mathbf{x}_{-u}')}{g_{\mathbf{X}_u,\mathbf{X}_{-u},\mathbf{X}_{-u}'}(\mathbf{x}_u,\mathbf{x}_{-u},\mathbf{x}_{-u}')}\\
&\hspace{8cm} g_{\mathbf{X}_u,\mathbf{X}_{-u},\mathbf{X}_{-u}'}(\mathbf{x}_u,\mathbf{x}_{-u},\mathbf{x}_{-u}') d\mathbf{x}_{-u}'d\mathbf{x}_{-u}d\mathbf{x}_{u}\\
&=\mathbb{E}_g\left(\psi_t\left(\mathbf{X}_u,\mathbf{X}_{-u} \right)\psi_t\left(\mathbf{X}_u,\mathbf{X}_{-u}' \right)\dfrac{f_{\mathbf{X}_u,\mathbf{X}_{-u},\mathbf{X}_{-u}'}(\mathbf{X}_u,\mathbf{X}_{-u},\mathbf{X}_{-u}')}{g_{\mathbf{X}_u,\mathbf{X}_{-u},\mathbf{X}_{-u}'}(\mathbf{X}_u,\mathbf{X}_{-u},\mathbf{X}_{-u}')} \right)\\
&=\mathbb{E}_g\left(\psi_t\left(\mathbf{X}_u,\mathbf{X}_{-u} \right)\psi_t\left(\mathbf{X}_u,\mathbf{X}_{-u}' \right)\dfrac{f_{\mathbf{X}}(\mathbf{X}_u,\mathbf{X}_{-u})\dfrac{f_{\mathbf{X}}(\mathbf{X}_u,\mathbf{X}_{-u}')}{f_{\mathbf{X}_u}(\mathbf{X}_u)}}{g(\mathbf{X}_u,\mathbf{X}_{-u})\dfrac{g(\mathbf{X}_u,\mathbf{X}_{-u}')}{g_{\mathbf{X}_u}(\mathbf{X}_u)}} \right)\\
&=\mathbb{E}_g\left(\psi_t\left(\mathbf{X} \right)\psi_t\left(\mathbf{X}^u \right)\dfrac{f_{\mathbf{X}}(\mathbf{X})\dfrac{f_{\mathbf{X}}(\mathbf{X}^u)}{f_{\mathbf{X}_u}(\mathbf{X}_u)}}{g(\mathbf{X})\dfrac{g(\mathbf{X}^u)}{g_{\mathbf{X}_u}(\mathbf{X}_u)}} \right)\\
&=\mathbb{E}_g\left(\psi_t\left(\mathbf{X} \right)\psi_t\left(\mathbf{X}^u \right)\dfrac{f_{\mathbf{X}}(\mathbf{X})f_{\mathbf{X}}(\mathbf{X}^u)g_{\mathbf{X}_u}(\mathbf{X}_u)}{g(\mathbf{X})g(\mathbf{X}^u)f_{\mathbf{X}_u}(\mathbf{X}_u) }\right).
\end{align*}
That concludes the proof of the \Cref{fund_is_pf}.\hfill $\square$

\section{Proofs of inequalities of section 3.3}
\subsection{Proof of inequality \eqref{eq_var_pf_gopt}}
\label{app:proof_eq_var_pf_gopt}

First, recall that $g_{\text{opt}}$ defined for all $\mathbf{x}\in\mathbb{X}$ by $g_{\text{opt}}\left(\mathbf{x}\right) = p_t^{-1}\psi_t\left(\mathbf{x}\right)f_{\mathbf{X}}\left(\mathbf{x}\right)$ is the optimal IS auxiliary density to estimate the failure probability by importance sampling, and its marginal PDF according to $\mathbf{X}_u$ is given by: \begin{equation}
    \forall \mathbf{x}_u \in \mathbb{X}_u, \ g_{\text{opt}_{\mathbf{X}_u}}(\mathbf{x}_u) = \dfrac{1}{p_t}\int_{\mathbb{X}_{-u}}f_{\mathbf{X}}\left(\mathbf{x}_u,\mathbf{x}_{-u}\right)\psi_t\left(\mathbf{x}_u,\mathbf{x}_{-u}\right) d\mathbf{x}_{-u}.
\end{equation} 
Therefore, when the considered IS auxiliary distribution is $g_{\text{opt}}$, the variance of the Pick-Freeze given-model estimator with importance sampling $\widehat{\text{T-VE}}_{u,\text{PF}}^{\text{IS}}$ in \eqref{tvengd} of $\text{T-VE}_u$ satisfies:
\begin{align*}
\mathbb{V}_{g_{\text{opt}}}\left(\widehat{\text{T-VE}}_{u,\text{PF}}^{\text{IS}}\right) &= \mathbb{V}_{g_{\text{opt}}}\left( \dfrac{1}{N_u}\sum_{n=1}^{N_u}w_t^{g_{\text{opt}}}\left(\mathbf{X}_u^{(n)},\mathbf{X}_{-u}^{(n,1)}\right)w_t^{g_{\text{opt}}}\left(\mathbf{X}_u^{(n)},\mathbf{X}_{-u}^{(n,2)}\right)\dfrac{g_{\text{opt}_{\mathbf{X}_u}}(\mathbf{X}_u^{(n)})}{f_{\mathbf{X}_u}(\mathbf{X}_u^{(n)})} - \underbrace{\widehat{p}_{t,N}^{\text{IS},\text{ub}}}_{p_t}\right)\\
&= \mathbb{V}_{g_{\text{opt}}}\left( \dfrac{1}{N_u}\sum_{n=1}^{N_u}w_t^{g_{\text{opt}}}\left(\mathbf{X}_u^{(n)},\mathbf{X}_{-u}^{(n,1)}\right)w_t^{g_{\text{opt}}}\left(\mathbf{X}_u^{(n)},\mathbf{X}_{-u}^{(n,2)}\right)\dfrac{g_{\text{opt}_{\mathbf{X}_u}}(\mathbf{X}_u^{(n)})}{f_{\mathbf{X}_u}(\mathbf{X}_u^{(n)})}\right) \\
&=\dfrac{1}{N_u}\mathbb{V}_{g_{\text{opt}}}\left(w_t^{g_{\text{opt}}}\left(\mathbf{X}\right) w_t^{g_{\text{opt}}}\left(\mathbf{X^u}\right) \dfrac{g_{\text{opt}_{\mathbf{X}_u}}(\mathbf{X}_u)}{f_{\mathbf{X}_u}(\mathbf{X}_u) }\right) \\ &=\dfrac{1}{N_u}\mathbb{V}_{g_{\text{opt}}}\left(\psi_t\left(\mathbf{X}\right) \psi_t\left(\mathbf{X}^u\right) \dfrac{f_{\mathbf{X}}(\mathbf{X})f_{\mathbf{X}}(\mathbf{X}^u)g_{\text{opt}_{\mathbf{X}_u}}(\mathbf{X}_u)}{g_{\text{opt}}(\mathbf{X})g_{\text{opt}}(\mathbf{X}^u)f_{\mathbf{X}_u}(\mathbf{X}_u) }\right) \\
&= \dfrac{1}{N_u}\mathbb{V}_{g_{\text{opt}}}\left(p_t^2\times \dfrac{1}{p_t}\dfrac{\int_{\mathbb{X}_{-u}}f_{\mathbf{X}}(\mathbf{X}_u,\mathbf{x}_{-u})\psi_t\left(\mathbf{X}_u,\mathbf{x}_{-u}\right) d\mathbf{x}_{-u} }{f_{\mathbf{X}_u}(\mathbf{X}_u)}\right) \\
&\hspace{2cm}\text{by integrating the exact expressions of $g_{\text{opt}}$ and $g_{\text{opt}_{\mathbf{X}_u}}$ given above}\\
&= \dfrac{1}{N_u}\mathbb{V}_{g_{\text{opt}}}\left(p_t\int_{\mathbb{X}_{-u}}\dfrac{f_{\mathbf{X}}(\mathbf{X}_u,\mathbf{x}_{-u})  }{f_{\mathbf{X}_u}(\mathbf{X}_u)}\psi_t\left(\mathbf{X}_u,\mathbf{x}_{-u}\right)d\mathbf{x}_{-u}\right)\\
&= \dfrac{p_t^2}{N_u}\mathbb{V}_{g_{\text{opt}}}\left(\int_{\mathbb{X}_{-u}}f_{\mathbf{X}_{-u}|\mathbf{X}_u}(\mathbf{x}_{-u})\psi_t\left(\mathbf{X}_u,\mathbf{x}_{-u}\right)d\mathbf{x}_{-u}\right)\\
&= \dfrac{p_t^2}{N_u}\mathbb{V}_{g_{\text{opt}}}\left[\mathbb{E}_f\left(\psi_t\left(\mathbf{X}\right) |\mathbf{X}_u\right) \right] \\
& \leq \dfrac{p_t^2}{N_u}.
\end{align*} That concludes the proof of inequality \eqref{eq_var_pf_gopt}.\hfill $\square$

\subsection{Proof of inequality \eqref{eq_var_pf_fx} }
\label{app:proof_eq_var_pf_fx}

The estimator $\widehat{\text{T-VE}}_{u,\text{PF}}$ by Pick-Freeze given-model without importance sampling of $\text{T-VE}_u$ is given by:\begin{equation}
    \widehat{\text{T-VE}}_{u,\text{PF}} = \dfrac{1}{N_u}\sum_{n=1}^{N_u}\psi_t\left(\mathbf{X}_u^{(n)},\mathbf{X}_{-u}^{(n,1)}\right)\psi_t\left(\mathbf{X}_u^{(n)},\mathbf{X}_{-u}^{(n,2)}\right) - \widehat{p}_{t,N}^2,
\end{equation} where $\widehat{p}_{t,N}$ is the empirical Monte Carlo estimator of $p_t$ and where the required samples are drawn according to the input distribution $f_{\mathbf{X}}$. In the given-model framework, independent samples are used to estimate the Pick-Freeze expectation $\mathbb{E}_{f_{\mathbf{X}}}\left[\psi_t(\mathbf{X})\psi_t(\mathbf{X}^u)\right]$ and the square failure probability $p_t^2$. Therefore, the variance of $\widehat{\text{T-VE}}_{u,\text{PF}}$ satisfies: \begin{align*}
        \mathbb{V}_{f_{\mathbf{X}}}\left(\widehat{\text{T-VE}}_{u,\text{PF}}\right) &= \mathbb{V}_{f_{\mathbf{X}}}\left(\dfrac{1}{N_u}\sum_{n=1}^{N_u}\psi_t\left(\mathbf{X}_u^{(n)},\mathbf{X}_{-u}^{(n,1)}\right)\psi_t\left(\mathbf{X}_u^{(n)},\mathbf{X}_{-u}^{(n,2)}\right) - \widehat{p}_{t,N}^2\right)\\
        &= \mathbb{V}_{f_{\mathbf{X}}}\left(\dfrac{1}{N_u}\sum_{n=1}^{N_u}\psi_t\left(\mathbf{X}_u^{(n)},\mathbf{X}_{-u}^{(n,1)}\right)\psi_t\left(\mathbf{X}_u^{(n)},\mathbf{X}_{-u}^{(n,2)}\right)\right) + \mathbb{V}_{f_{\mathbf{X}}}\left(\widehat{p}_{t,N}^2\right)\\
        &\geq \mathbb{V}_{f_{\mathbf{X}}}\left(\dfrac{1}{N_u}\sum_{n=1}^{N_u}\psi_t\left(\mathbf{X}_u^{(n)},\mathbf{X}_{-u}^{(n,1)}\right)\psi_t\left(\mathbf{X}_u^{(n)},\mathbf{X}_{-u}^{(n,2)}\right)\right)\\
        &= \frac{1}{N_u} \mathbb{V}_{f_{\mathbf{X}}}\left(\psi_t(\mathbf{X})\psi_t(\mathbf{X}^u)\right) \\
        &= \frac{1}{N_u}\mathbb{E}_{f_{\mathbf{X}}}\left(\psi_t(\mathbf{X})\psi_t(\mathbf{X}^u)\right) - \frac{1}{N_u}\mathbb{E}_{f_{\mathbf{X}}}\left(\psi_t(\mathbf{X})\psi_t(\mathbf{X}^u)\right)^2\\
        &= \frac{1}{N_u}\left(\mathbb{V}_{f_{\mathbf{X}}}\left[\mathbb{E}_{f_{\mathbf{X}}}\left(\psi_t(\mathbf{X}) | \mathbf{X}_u\right)\right] + p_t^2\right) - \frac{1}{N_u}\left(\mathbb{V}_{f_{\mathbf{X}}}\left[\mathbb{E}_{f_{\mathbf{X}}}\left(\psi_t(\mathbf{X}) | \mathbf{X}_u\right)\right] + p_t^2\right)^2\\
        & \hspace{8cm} \text{thanks to \eqref{PF_fiab_dep}}\\
        &= \frac{1}{N_u}\left(\mathbb{V}_{f_{\mathbf{X}}}\left[\mathbb{E}_{f_{\mathbf{X}}}\left(\psi_t(\mathbf{X}) | \mathbf{X}_u\right)\right] + p_t^2\right)\left(1-\left(\mathbb{V}_{f_{\mathbf{X}}}\left[\mathbb{E}_{f_{\mathbf{X}}}\left(\psi_t(\mathbf{X}) | \mathbf{X}_u\right)\right] + p_t^2\right)\right).
\end{align*} That concludes the proof of inequality \eqref{eq_var_pf_fx}.\hfill $\square$

\section{Theoretical values of the target Shapley effects in the Gaussian linear framework}
\label{app:theo_lg}

Let us consider the Gaussian linear framework introduced in \Cref{sec:linear_gaussian}, and a failure threshold $t\in\mathbb{R}$. Recall that the input covariance matrix is symmetric positive-definite and that ${\boldsymbol{\beta}} \neq 0$, thus we have ${\boldsymbol{\beta}}^\top \mathbf{\Sigma} {\boldsymbol{\beta}} >0$.

Moreover, let us recall the following theorem:\begin{theorem}\label{com_lin_gauss}For $k\geq 1$, if $\mathbf{A}\in \mathcal{M}_{k,d}\left(\mathbb{R}\right)$, $\mathbf{b}\in\mathbb{R}^k$ and $\mathbf{X}\sim\mathcal{N}_d\left(\boldsymbol{\mu},\mathbf{\Sigma}\right)$, then:\begin{equation}
    \mathbf{A}\mathbf{X} + \mathbf{b} \sim \mathcal{N}_k\left(\mathbf{A}\boldsymbol{\mu} + \mathbf{b}, \mathbf{A}\mathbf{\Sigma}\mathbf{A}^\top\right).
\end{equation}
\end{theorem}

\subsection{Theoretical value of the failure probability}
\label{app:theo_fail_prob_lg}
\begin{theorem}
The failure probability is given by: \begin{equation}
    p_t^{\boldsymbol{\beta}} = 1 - \Phi\left(\dfrac{t-{\boldsymbol{\beta}}^\top\boldsymbol{\mu}}{\sqrt{{\boldsymbol{\beta}}^\top \mathbf{\Sigma} {\boldsymbol{\beta}}}} \right).
\end{equation}
\end{theorem}

\begin{proof}
In the Gaussian linear framework, the failure probability satisfies:
\begin{equation}p_t^{\boldsymbol{\beta}} = \mathbb{P}\left(\phi_{\boldsymbol{\beta}}\left(\mathbf{X}\right)>t\right)=\mathbb{P}\left({\boldsymbol{\beta}}^\top\mathbf{X}>t\right) = \mathbb{P}\left(\dfrac{{\boldsymbol{\beta}}^\top\mathbf{X}-{\boldsymbol{\beta}}^\top\boldsymbol{\mu}}{\sqrt{{\boldsymbol{\beta}}^\top \mathbf{\Sigma} {\boldsymbol{\beta}}}}>\dfrac{t-{\boldsymbol{\beta}}^\top\boldsymbol{\mu}}{\sqrt{{\boldsymbol{\beta}}^\top \mathbf{\Sigma} {\boldsymbol{\beta}}}}\right),
\end{equation} with ${\boldsymbol{\beta}}^\top \mathbf{\Sigma} {\boldsymbol{\beta}}>0$. Then, \Cref{com_lin_gauss} provides that $\left({\boldsymbol{\beta}}^\top\mathbf{X}-{\boldsymbol{\beta}}^\top\boldsymbol{\mu}\right)\left/\sqrt{{\boldsymbol{\beta}}^\top \mathbf{\Sigma} {\boldsymbol{\beta}}}\right. \sim \mathcal{N}_1\left(0,1\right)$. Finally, we have:\begin{equation}\label{proba_theo}
    p_t^{\boldsymbol{\beta}} = 1 - \Phi\left(\dfrac{t-{\boldsymbol{\beta}}^\top\boldsymbol{\mu}}{\sqrt{{\boldsymbol{\beta}}^\top \mathbf{\Sigma} {\boldsymbol{\beta}}}} \right),
\end{equation} where $\Phi$ is the CDF of the 1-dimensional standard Normal distribution.
\end{proof}

\subsection{Theoretical values of the target closed Sobol indices}
\label{app:theo_tcs_lg}

\begin{theorem}
For $u\in\mathcal{P}(d)\backslash\lbrace\varnothing,[\![1,d]\!]\rbrace$, the target closed Sobol index is given by: \begin{equation}\label{soboltheo}\text{T-VE}_u = \left\{
    \begin{array}{ll}
        \mathbb{V}\left[\Phi\left(\dfrac{t-{\boldsymbol{\beta}}_{u}^\top\mathbf{X}_{u} - {\boldsymbol{\beta}}_{-u}^\top\left(\boldsymbol{\mu}_{-u} + \mathbf{\Sigma}_{-u,u}\mathbf{\Sigma}_{u,u}^{-1}\left(\mathbf{X}_u - \boldsymbol{\mu}_u \right) \right)}{\sqrt{{\boldsymbol{\beta}}_{-u}^\top\left(\mathbf{\Sigma}_{-u,-u} - \mathbf{\Sigma}_{-u,u}\mathbf{\Sigma}_{u,u}^{-1}\mathbf{\Sigma}_{u,-u} \right)  {\boldsymbol{\beta}}_{-u}}} \right) \right]   & \mbox{ if } {\boldsymbol{\beta}}_{-u} \neq 0 \\
        p_t^{\boldsymbol{\beta}}\left( 1 - p_t^{\boldsymbol{\beta}}\right) & \mbox{else,}
    \end{array}
\right.
\end{equation} where for $u_1,u_2\in\mathcal{P}(d)$, $\mathbf{\Sigma}_{u_1,u_2} = \left(\Sigma_{i,j}\right)_{i\in u_1, j \in u_2}$. At last, using  the definition in \eqref{shapley}, one can derive the theoretical values of the target Shapley effects in the Gaussian linear framework.
\end{theorem}

\begin{proof}For any subset $u\in\mathcal{P}(d)\backslash\lbrace\varnothing,[\![1,d]\!]\rbrace$, let us compute the theoretical value of $\text{T-VE}_u$ in the Gaussian linear framework.


\textbullet \quad Case 1: ${\boldsymbol{\beta}}_{-u}\neq 0$ \\ In that case, we have: \begin{align*}
\mathbb{E}\left(\mathbf{1}\left(\phi_{\boldsymbol{\beta}}\left(\mathbf{X} \right)>t\right)  | \mathbf{X}_u = \mathbf{x}_u \right) &= \mathbb{P}\left(\phi_{\boldsymbol{\beta}}\left(\mathbf{X} \right)>t | \mathbf{X}_u = \mathbf{x}_u\right)  \\
&=\mathbb{P}\left({\boldsymbol{\beta}}^\top\mathbf{X}>t | \mathbf{X}_u = \mathbf{x}_u\right)\\
&= \mathbb{P}\left({\boldsymbol{\beta}}_{-u}^\top\mathbf{X}_{-u}>t-{\boldsymbol{\beta}}_{u}^\top\mathbf{x}_{u} | \mathbf{X}_u = \mathbf{x}_u\right).
\end{align*}
Then, recall that the conditional normal distribution satisfies: \begin{equation}\mathbf{X}_{-u} | \mathbf{X}_u = \mathbf{x}_u \sim \mathcal{N}_{|-u|}\left(\boldsymbol{\mu}_{-u} + \mathbf{\Sigma}_{-u,u}\mathbf{\Sigma}_{u,u}^{-1}\left(\mathbf{x}_u - \boldsymbol{\mu}_u \right),  \mathbf{\Sigma}_{-u,-u} - \mathbf{\Sigma}_{-u,u}\mathbf{\Sigma}_{u,u}^{-1}\mathbf{\Sigma}_{u,-u} \right).\end{equation}
Therefore, thanks to \Cref{com_lin_gauss}, the conditional distribution of ${\boldsymbol{\beta}}_{-u}^\top\mathbf{X}_{-u} | \mathbf{X}_u = \mathbf{x}_u$ is given by: \begin{multline}{\boldsymbol{\beta}}_{-u}^\top\mathbf{X}_{-u} | \mathbf{X}_u = \mathbf{x}_u \sim \\ \mathcal{N}_1\left[{\boldsymbol{\beta}}_{-u}^\top\left(\boldsymbol{\mu}_{-u} + \mathbf{\Sigma}_{-u,u}\mathbf{\Sigma}_{u,u}^{-1}\left(\mathbf{x}_u - \boldsymbol{\mu}_u \right) \right), {\boldsymbol{\beta}}_{-u}^\top\left(\mathbf{\Sigma}_{-u,-u} - \mathbf{\Sigma}_{-u,u}\mathbf{\Sigma}_{u,u}^{-1}\mathbf{\Sigma}_{u,-u} \right)  {\boldsymbol{\beta}}_{-u}  \right].\end{multline}
Finally, we have: 
\begin{equation}\mathbb{E}\left(\mathbf{1}\left(\phi_{\boldsymbol{\beta}}\left(\mathbf{X}\right)>t\right)  | \mathbf{X}_u = \mathbf{x}_u \right) = 1 - \Phi\left(\dfrac{t-{\boldsymbol{\beta}}_{u}^\top\mathbf{x}_{u} - {\boldsymbol{\beta}}_{-u}^\top\left(\boldsymbol{\mu}_{-u} + \mathbf{\Sigma}_{-u,u}\mathbf{\Sigma}_{u,u}^{-1}\left(\mathbf{x}_u - \boldsymbol{\mu}_u \right) \right)}{\sqrt{{\boldsymbol{\beta}}_{-u}^\top\left(\mathbf{\Sigma}_{-u,-u} - \mathbf{\Sigma}_{-u,u}\mathbf{\Sigma}_{u,u}^{-1}\mathbf{\Sigma}_{u,-u} \right)  {\boldsymbol{\beta}}_{-u}}} \right),\end{equation} an so: \begin{equation}
    \text{T-VE}_u = \mathbb{V}\left[\Phi\left(\dfrac{t-{\boldsymbol{\beta}}_{u}^\top\mathbf{X}_{u} - {\boldsymbol{\beta}}_{-u}^\top\left(\boldsymbol{\mu}_{-u} + \mathbf{\Sigma}_{-u,u}\mathbf{\Sigma}_{u,u}^{-1}\left(\mathbf{X}_u - \boldsymbol{\mu}_u \right) \right)}{\sqrt{{\boldsymbol{\beta}}_{-u}^\top\left(\mathbf{\Sigma}_{-u,-u} - \mathbf{\Sigma}_{-u,u}\mathbf{\Sigma}_{u,u}^{-1}\mathbf{\Sigma}_{u,-u} \right)  {\boldsymbol{\beta}}_{-u}}} \right) \right].
\end{equation}


\textbullet \quad Case 2: ${\boldsymbol{\beta}}_{-u} = 0$ \\ In that case, the linear function $\phi_{\boldsymbol{\beta}}$ depends only on $\mathbf{x}_u$. Then: \begin{equation} \mathbb{E}\left(\mathbf{1}\left(\phi_{\boldsymbol{\beta}}(\mathbf{X})>t\right)  | \mathbf{X}_u \right) = \mathbf{1}\left(\phi_{\boldsymbol{\beta}}(\mathbf{X})>t\right),\end{equation} and finally: \begin{equation}\text{T-VE}_u = \mathbb{V}\left[\mathbb{E}\left(\mathbf{1}\left(\phi_{\boldsymbol{\beta}}(\mathbf{X})>t\right)  | \mathbf{X}_u \right)  \right] = \mathbb{V}\left(\mathbf{1}\left(\phi_{\boldsymbol{\beta}}(\mathbf{X})>t\right) \right) = p_t^{\boldsymbol{\beta}}\left(1-p_t^{\boldsymbol{\beta}}\right).\end{equation}

To sum up, we have just proved the required result in \eqref{soboltheo} .
\end{proof}

\section{Preprocessing procedure for the given-data estimators}
\label{app:standardisation}

\subsection{Presentation of the procedure}
\label{ss:preprocessing}
The preprocessing procedure presented below will be applied as soon as a given-data estimator will be used, and it is based on the following theorem stated by \cite{owen2017shapley}. \begin{theorem}\label{invariance_bij}
Consider a family of uni-dimensional bijective transformations $\left(\tau_i\right)_{i\in[\![1,d]\!]}$. Let us define the following function: \begin{equation}
    \begin{array}{l|rcl}
\widetilde{\phi} : & \bigotimes_{i=1}^d \tau_i\left(\mathbb{X}_i\right) & \longrightarrow & \mathbb{R} \\
    & \mathbf{z} & \longmapsto & \phi\left(\tau_1^{-1}\left(z_1\right),\dots,\tau_d^{-1}\left(z_d\right)\right),
    \end{array} 
\end{equation} its random input vector $\mathbf{Z} = \left(\tau_i\left(X_i\right)\right)_{i\in[\![1,d]\!]}$ and let us write $\left(\widetilde{\text{Sh}}_i\right)_{i\in[\![1,d]\!]}$ its Shapley effects. Recalling that $\left(\text{Sh}_i\right)_{i\in[\![1,d]\!]}$ are the Shapley effects of $\phi$, then: \begin{equation}
    \forall i \in [\![1,d]\!], \ \text{Sh}_i = \widetilde{\text{Sh}}_i.
\end{equation}
\end{theorem} In other words, this theorem shows that the Shapley effects are unchanged when bijective transformations are applied on each input variable. Practically, the preprocessing consists in applying the following procedure: \begin{enumerate}
    \item choose a sampling distribution $h\in\lbrace f_{\mathbf{X}},g\rbrace$
    \item if it is possible, for all $i\in[\![1,d]\!]$, compute the exact values of $\mu^{(h)}_i = \mathbb{E}_h\left(X_i\right)$ and $\left(\sigma_i^{(h)}\right)^2 = \mathbb{V}_h\left(X_i\right)$, else estimate them
    \item for all $i\in[\![1,d]\!]$, define the linear bijective transformations by:\begin{equation}\label{bij_trans}
    \forall x_i\in\mathbb{R}, \ \tau_i\left(x_i\right) = \dfrac{x_i - \mu^{(h)}_i}{\sigma_i^{(h)}}
\end{equation}
    \item from a sample $\left(\mathbf{X}^{(n)}\right)_{n\in[\![1,N]\!]}$ drawn according to $h$, build the transformed sample $\left(\mathbf{Z}^{(n)}\right)_{n\in[\![1,N]\!]}$ defined for all $n\in[\![1,N]\!]$ by $\mathbf{Z}^{(n)} = \left(\tau_i\left(X_i^{(n)}\right)\right)_{i\in[\![1,d]\!]}$
    \item apply the previous given-data estimators proposed in this article to the new sample $\left(\mathbf{Z}^{(n)}\right)_{n\in[\![1,N]\!]}$.
\end{enumerate} Hence, \Cref{invariance_bij} applied to $\psi_t : \mathbf{x} \in \mathbb{X}\mapsto\mathbf{1}\left(\phi\left(\mathbf{x}\right)>t\right)$ with the family of linear bijective transformations $\left(\tau_i\right)_{i\in[\![1,d]\!]}$ defined in \eqref{bij_trans} justifies that the described preprocessing procedure should theoretically provide the expected target Shapley effects. Moreover, the transformations $\left(\tau_i\right)_{i\in[\![1,d]\!]}$ defined in \eqref{bij_trans} only consist here in a standardisation of the input sample $\left(\mathbf{X}^{(n)}\right)_{n\in[\![1,N]\!]}$ drawn according to $h\in\lbrace f_{\mathbf{X}},g\rbrace$, i.e. a re-scaling by $\sqrt{\mathbb{V}_h\left(X_i\right)}$ and a shifting by $\mathbb{E}_h\left(X_i\right)$ of each component of each point. In particular, the nearest-neighbour search is performed among the new sample $\left(\mathbf{Z}^{(n)}\right)_{n\in[\![1,N]\!]}$ in which distances between points are expected to be more homogeneous.

\subsection{Motivations and practical interest}

In order to motivate the introduction of the preprocessing procedure presented in \ref{ss:preprocessing}, let us reconsider the cantilever beam example of \Cref{ss:cantilever_beam}. If we estimate the target Shapley effects of this problem without applying the preprocessing procedure, we obtain the results presented in \Cref{fig:cantilever_beam}. The existing given-data estimators without importance sampling, especially those of the second, fourth and sixth indices, seem to be badly biased. This phenomenon can perhaps be explained by the huge scale difference between the third variable, the elastic modulus $E$, and the others. When $3$ is in a subset $u\in \mathcal{P}(6)\backslash\lbrace\varnothing,[\![1,6]\!]\rbrace$, the distance in the subspace $\mathbb{X}_u$ is approximately equal to the distance in $\mathbb{X}_{\left\lbrace 3 \right\rbrace}$, which makes the nearest neighbour approximation of a conditional distribution given some $\mathbf{x}_u \in \mathbb{X}_u$ very inaccurate. This phenomenon is getting worse without importance sampling because the points of interest, the failure points, are in the tail of the distribution, where the concentration of points is small and thus where the distances between points are even larger. 

Consequently, the preprocessing procedure described in \ref{ss:preprocessing} aims to restructure the available sample such that each component has the same scale in order to decrease the error due to the nearest neighbour approximation. By definition, for a given $u\in \mathcal{P}(d)\backslash\lbrace\varnothing,[\![1,d]\!]\rbrace$, this error mainly comes from the gap between $\psi_t\left(\mathbf{X}^{(n)}_{u},\mathbf{X}^{(k_N^u(n,2))}_{-u}\right)$, the target value on the subspace $\left\lbrace \mathbf{x} \in \mathbb{R}^d / \mathbf{x}_u = \mathbf{X}^{(n)}_{u}\right\rbrace$, and $\psi_t\left(\mathbf{X}^{(k_N^u(n,2))}\right)$ which is its approximation, for all $n\in[\![1,N_u]\!]$. The restructuring caused by the linear bijective transformations defined in \eqref{bij_trans} aims then at reducing this error. Indeed, the nearest neighbour of some points of the new sample $\left(\mathbf{Z}^{(n)}\right)_{n\in[\![1,N]\!]}$ might not be the same as in the starting sample $\left(\mathbf{X}^{(n)}\right)_{n\in[\![1,N]\!]}$. \Cref{fig:bij_trans} illustrates this phenomenon. More precisely, for all $u\in \mathcal{P}(d)\backslash\lbrace\varnothing,[\![1,d]\!]\rbrace$, there can exist some $n_0 \in [\![1,N]\!]$ such that $k_{\mathbf{x}}^u(n_0,2) \neq k_{\mathbf{z}}^u(n_0,2)$, where $k_{\mathbf{x}}^u(n_0,2)$ and $k_{\mathbf{z}}^u(n_0,2)$ represent the indices of the second nearest neighbour of the point $n_0$ respectively in the $\mathbf{x}$-space and in the $\mathbf{z}$-space, and thus $\psi_t\left(\mathbf{X}^{(k_{\mathbf{x}}^u(n_0,2))}\right) \neq \widetilde{\psi_t}\left(\mathbf{Z}^{(k_{\mathbf{z}}^u(n_0,2))}\right)$. At last, we expect that most of the time, this change leads to a reduction of the error, i.e. $\left|\widetilde{\psi_t}\left(\mathbf{Z}^{(k_{\mathbf{z}}^u(n_0,2))}\right) - \widetilde{\psi_t}\left(\mathbf{Z}^{(n_0)}_{u},\mathbf{Z}^{(k_{\mathbf{z}}^u(n_0,2))}_{-u}\right)\right| \leq \left|\psi_t\left(\mathbf{X}^{(k_{\mathbf{x}}^u(n_0,2))}\right) - \psi_t\left(\mathbf{X}^{(n_0)}_{u},\mathbf{X}^{(k_{\mathbf{x}}^u(n_0,2))}_{-u}\right)\right|$. A more advanced theoretical study is required to better understand this preprocessing procedure and potentially to improve it, but it seems beneficial in our examples, especially on the cantilever beam example. At last, note that in \Cref{sec:linear_gaussian}, we apply the preprocessing procedure with the given-data estimators. However, we would obtain almost exactly the same results without having applied it because in those examples, the transformations have a very mild impact.

\begin{figure}
    \centering
    \includegraphics[width=.6\textwidth]{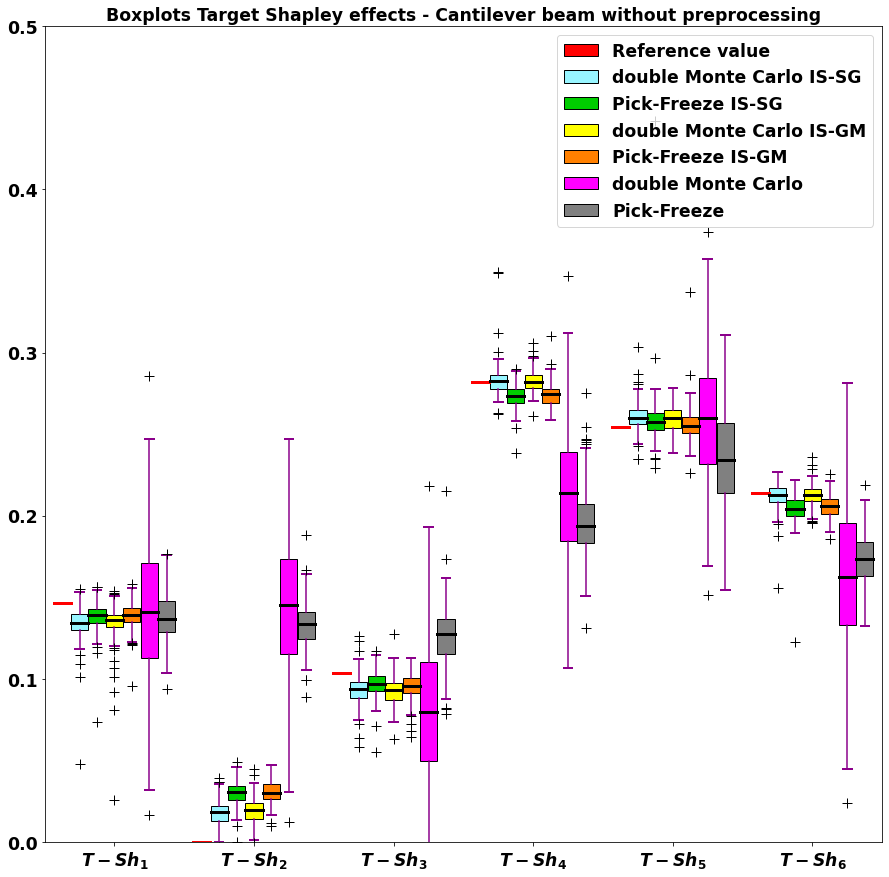}
    \caption{Estimation of the target Shapley effects in the cantilever beam example, in the given-data framework and without the preprocessing described in \ref{ss:preprocessing}.}
    \label{fig:cantilever_beam}
\end{figure}

\begin{figure}
\centering
\includegraphics[width=.6\textwidth]{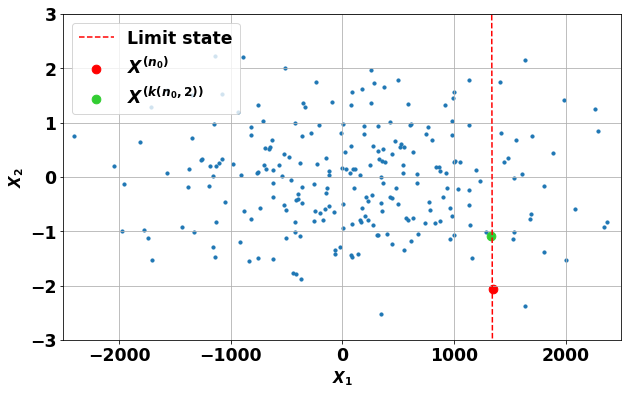}
\includegraphics[width=.6\textwidth]{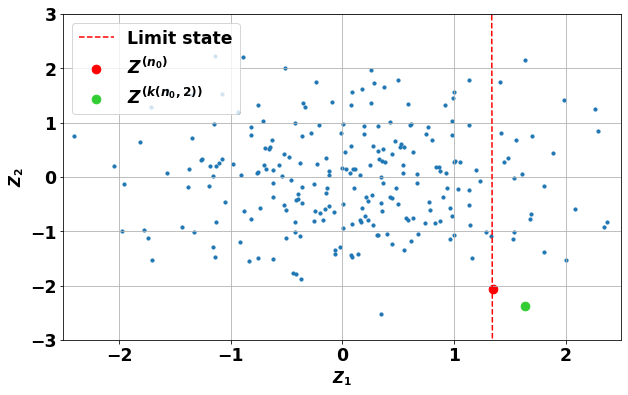}
\caption{The upper figure represents in the $\mathbf{x}$-space a sample of $N=200$ points drawn according to a zero-mean bi-dimensional normal distribution with independent components and such that $\mathbb{V}\left(X_1\right) = 10^9$ and $\mathbb{V}\left(X_2\right) = 1$. The lower figure represents in the $\mathbf{z}$-space the same sample after having applied on each component the transformations defined in \eqref{bij_trans}. On both figures, the dotted red line represents the limit state of the linear function $\phi_{\left( 1 \ 1 \right)}\left(x_1,x_2\right) = x_1 + x_2$ defined by the failure threshold $t = 1350$, the red point represents the point $n_0\in[\![1,200]\!]$ and the green point represents its nearest neighbour among the corresponding sample with the $2$-dimensional distance. In both cases, the red point is in the failure domain. Moreover, in the $\mathbf{x}$-space (upper figure), the green point is in the safe domain whereas in the $\mathbf{z}$-space (lower figure), the green point is in the failure domain.}
\label{fig:bij_trans}
\end{figure}


\section{Equations of the fire spread model in Section 4.3}
\label{app:fire_spread_eqs}

Given the random input vector $\mathbf{X} = \left(\delta,\sigma,h,\rho_p,m_l,m_d,S_T,U,\mathrm{tan} \ \varphi,P\right)$, the rate of spread is obtained through the following system of equations: 

$$\begin{array}{llr}
    &R =\dfrac{I_R\xi\left(1+\phi_W+\phi_S\right)}{\rho_b\epsilon Q_{ig}}&\text{rate of fire-spread, ft}\cdot\text{min}^{-1}\\
    &  \hspace{-.5cm}\text{where} \notag \\
    &w_0 = \dfrac{4.8}{4.8824} \times \dfrac{1}{1+\exp\left[(15-\delta)/3.5\right]}&\text{fuel loading, kg}\cdot\text{m}^{-2} \\
    &\Gamma_{\text{max}} = \dfrac{\sigma^{1.5}}{495+0.0594\sigma^{1.5}}&\text{maximum reaction velocity, min}^{-1}\\
    &\beta_{\text{op}} = 3.348\sigma^{-0.8189}&\text{optimum packing ratio}\\
    &A = 133\sigma^{-0.7913}\\
    &\theta^* = \dfrac{301.4 - 305.87(m_l-m_d)+2260m_d}{2260m_l}\\
    &\theta = \mathrm{min}\left(1,\mathrm{max}\left(0,\theta^*\right)\right)\\
    &\mu_M = \exp\left[-7.3Pm_d - (7.3\theta+2.13)(1-P)m_l\right]&\text{moisture damping coefficient}\\
    &\mu_S = 0.174S_T^{-0.19}&\text{mineral damping coefficient}\\
    &C = 7.47\exp\left[-0.133\sigma^{0.55}\right]\\
    &B = 0.02526\sigma^{0.54}\\
    &E = 0.715\exp\left[-3.59\times10^{-4}\sigma\right]\\
    &w_n = w_0(1-S_T)&\text{net fuel loading, lb}\cdot\text{ft}^{-2}\\
    &\rho_b = \dfrac{w_0}{\delta}&\text{ovendry bulk density, lb}\cdot\text{ft}^{-3}\\
    &\epsilon = \exp\left[\dfrac{-138}{\sigma}\right]&\text{effective heating number}\\
    &Q_{ig} = 130.87 + 1054.43m_d&\text{heat of preignition, Btu}\cdot\text{lb}^{-1}\\
    &\beta = \dfrac{\rho_b}{\rho_p}&\text{packing ratio}\\
    &\Gamma = \Gamma_{\text{max}}\left(\dfrac{\beta}{\beta_{\text{op}}}\right)^A\exp\left[A\left(1-\dfrac{\beta}{\beta_{\text{op}}}\right)\right]&\text{optimum reaction velocity, min}^{-1}\\
    &\xi = \dfrac{\exp\left[\left(0.792+0.681\sigma^{0.5}\right)\left(\beta+0.1\right)\right]}{192+0.2595\sigma}&\text{propagating flux ratio}\\
    &\phi_W = CU^B\left(\dfrac{\beta}{\beta_{\text{op}}}\right)^{-E}&\text{wind coefficient}\\
    &\phi_S = 5.275\beta^{-0.3}\left(\mathrm{tan} \ \varphi\right)^2&\text{slope factor}\\
    &I_R = \Gamma w_n h \mu_M \mu_S&\text{reaction intensity, Btu}\cdot\text{ft}^{-2}\cdot\text{min}^{-1}.
\end{array}$$ Note that the expression of the fuel loading $w_0$ according to the fuel depth $\delta$ is conjectured from the data analysis performed in \cite{salvador2001global}. In addition, it is important to remark that almost all the above equations, which mainly come from \cite{rothermel1972mathematical}, are given in imperial units whereas the inputs are specified in metric units in \Cref{table:distr_fire_spread}. In order to have consistent results, it is thus necessary to convert the input variables into the imperial units at the beginning of the numerical calculus and to convert the output into $\mathrm{cm\cdot s}^{-1}$ at the end.

\section{Cost-reduction estimation procedure in the given-model framework}
\label{sss:3.3.2}

In practice, the ROSA of a complex system always comes after the reliability analysis, i.e. the estimation of the failure probability. The given-model estimators \eqref{tevngd} and \eqref{tvengd} of the target conditional indices and thus the corresponding target Shapley effect estimators have a high computational cost, and the reliability analysis provides an i.i.d. input/output $N$-sample $\left(\mathbf{X}^{(n)},\psi_t\left(\mathbf{X}^{(n)}\right)\right)_{n\in[\![1,N]\!]}$ with $\left(\mathbf{X}^{(n)}\right)_{n\in[\![1,N]\!]}$ distributed according to the IS auxiliary density $g$. We present here a new procedure which re-uses the available sample in order to reduce the computational cost of the previous given-model estimators of the target conditional indices $\text{T-VE}_u$ and $\text{T-EV}_u$ and thus of the target Shapley effects. Remark first that estimators by importance sampling of $p_t$ and $\mathbb{V}_{f_{\mathbf{X}}}\left(\mathbf{1}\left(\phi\left(\mathbf{X}\right)>t\right)\right)$ can be easily computed with only the available sample and so do not require additional calls to $\phi$.

\subsection{Double Monte Carlo procedure}

With the double Monte Carlo method in \eqref{tevngd}, for $u\in\mathcal{P}(d)\backslash\lbrace\varnothing,[\![1,d]\!]\rbrace$, set the parameters $N_u$ and $N_I$ and consider a sequence $\left(s(n)\right)_{n\in[\![1,N_u]\!]}$ of uniformly distributed integers in $[\![1,N]\!]$ and apply the following scheme: \begin{enumerate}
    \item for $n\in[\![1,N_u]\!]$, draw an i.i.d. sample $\left(\widetilde{\mathbf{X}}_{u}^{(s(n),2)},\dots,\widetilde{\mathbf{X}}_{u}^{(s(n),N_I)}\right)$ distributed according to the conditional distribution of $g_{\mathbf{X}_{u}|\mathbf{X}_{-u} = \mathbf{X}_{-u}^{(s(n))}}$
    \item for $k\in[\![2,N_I]\!]$, compute $w_t^g\left(\widetilde{\mathbf{X}}_u^{(s(n),k)},\mathbf{X}_{-u}^{(s(n))}\right)$ and use the value of $w_t^g\left(\mathbf{X}^{(s(n))}\right)$ for the term corresponding to $k=1$
    \item compute the estimators \eqref{bias_est} and then \eqref{tevngd}.
\end{enumerate} This procedure requires $N_u\left(N_I-1\right)$ additional calls to $\phi$ to those from the reliability analysis to estimate the target conditional index $\text{T-EV}_u$ by double Monte Carlo with importance sampling. Typically, the authors of \cite{song2016shapley} suggest to use $N_I = 3$ for numerical purposes, thus the proposed procedure is interesting because it does not require too many additional calls to the code $\phi$.

\subsection{Pick-Freeze procedure}

With the Pick-Freeze method in \eqref{tvengd}, for any subset $u\in\mathcal{P}(d)\backslash\lbrace\varnothing,[\![1,d]\!]\rbrace$, set the parameter $N_u$, consider a sequence $\left(s(n)\right)_{n\in[\![1,N_u]\!]}$ of uniformly distributed integers in $[\![1,N]\!]$ and apply the following scheme: \begin{enumerate}
    \item for $n\in[\![1,N_u]\!]$, draw a random variable $\widetilde{\mathbf{X}}_{-u}^{(s(n),2)}$ from the conditional distribution of $g_{\mathbf{X}_{-u}|\mathbf{X}_{u} = \mathbf{X}_{u}^{(s(n))}}$
    \item compute $w_t^g\left(\mathbf{X}_u^{(s(n))},\widetilde{\mathbf{X}}_{-u}^{(s(n),2)}\right)$ and use the value of $w_t^g\left(\mathbf{X}^{(s(n))}\right)$ for the first term in the Pick-Freeze product
    \item compute the estimator \eqref{tvengd}.
\end{enumerate} This procedure requires $N_u$ additional calls to $\phi$ to those from the reliability analysis to estimate the target conditional index $\text{T-VE}_u$ by Pick-Freeze with importance sampling. 

\subsection{Cost reduction provided}
In the above cost-reduction procedure with the double Monte Carlo (resp. Pick-Freeze) method, the sampling from the marginal distribution $g_{\mathbf{X}_{-u}}$ (resp. $g_{\mathbf{X}_{u}}$) is replaced by picking a random sub-sample among the sample $\left(\mathbf{X}_{-u}^{(n)}\right)_{n\in[\![1,N]\!]}$ (resp. $\left(\mathbf{X}_{u}^{(n)}\right)_{n\in[\![1,N]\!]}$) through the random sequence $\left(s(n)\right)_{n\in[\![1,N_u]\!]}$. Then, for any $n\in[\![1,N_u]\!]$, a sample of size $N_I-1$ (resp. $1$) is drawn according to the conditional distribution $g_{\mathbf{X}_{u}|\mathbf{X}_{-u} = \mathbf{X}_{-u}^{(s(n))}}$ (resp. $g_{\mathbf{X}_{-u}|\mathbf{X}_{u} = \mathbf{X}_{u}^{(s(n))}}$) and the missing point is set to $\mathbf{X}_u^{(s(n))}$ (resp. $\mathbf{X}_{-u}^{(s(n))}$) such that we obtain an i.i.d. sample of size $N_I$ (resp. 2) distributed according to the corresponding conditional distribution. The latter missing point belongs to the available sample and does not require to be evaluated and so allows to save one call to $\phi$. Eventually, given the data from the reliability analysis, this new procedure allows to save $N_u$ calls to $\phi$ to estimate each target conditional index and so allows to save $N_{\mathbb{V}} + m(d-1)N_O$ calls to $\phi$ to estimate the $d$ target Shapley effects with the random permutation aggregation procedure and $N_{\mathbb{V}} + (2^d-2)N_O$ calls with the subset aggregation procedure, where $N_{\mathbb{V}}$ and $N_O$ are defined in \Cref{ss:aggreg}.









\bibliographystyle{IJ4UQ_Bibliography_Style}
\bibliography{References}
\end{document}